\documentclass{amsart}

\usepackage{amsmath}
\usepackage{amsthm}
\usepackage{enumerate}
\usepackage{url}

\usepackage{bm}
\usepackage{upgreek}

\usepackage{booktabs}

\usepackage{mathrsfs}

\usepackage{amssymb}

\usepackage{tikz-cd}

\usepackage{mathtools}

\usepackage{CJK}

\begin{CJK}{UTF8}{min}
\gdef\san{\mbox{\bf 山}}
\gdef\ten{\mbox{\bf 天}}
\end{CJK}

\newcommand{\C}{\mathbb{C}}
\newcommand{\F}{\mathbb{F}}
\newcommand{\Q}{\mathbb{Q}}
\newcommand{\R}{\mathbb{R}}
\newcommand{\Z}{\mathbb{Z}}
\newcommand{\SQ}{S_\mathbb{Q}}
\newcommand{\OS}{\mathscr{O}_S}
\newcommand{\OSQ}{\mathscr{O}_{S_\mathbb{Q}}}
\newcommand{\OK}{\mathscr{O}_K}

\newcommand{\OT}{\mathscr{O}_T}

\newcommand{\bmu}{\bm{\upmu}}
\newcommand{\brho}{\bm{\uprho}}
\newcommand{\bDelta}{\bm{\Delta}}

\renewcommand{\Re}{\mathrm{Re}\,}
\renewcommand{\Im}{\mathrm{Im}\,}
\renewcommand{\P}{\mathbb{P}}

\DeclareMathOperator{\ord}{ord}
\DeclareMathOperator{\lcm}{lcm}
\DeclareMathOperator{\Gal}{Gal}
\DeclareMathOperator{\Out}{Out}
\DeclareMathOperator{\Aut}{Aut}

\newtheorem{theorem}{Theorem}[section]
\newtheorem*{main_theorem*}{Main Theorem}
\newtheorem{lemma}[theorem]{Lemma}
\newtheorem{proposition}[theorem]{Proposition}
\newtheorem{corollary}[theorem]{Corollary}

\theoremstyle{definition}
\newtheorem*{remark}{Remark}

\newtheorem*{example}{Example}

\title[Picard Curves over $\Q$ good outside $3$]{Picard Curves over $\Q$ with \\ good reduction away from $3$}

\author{Beth Malmskog and Christopher Rasmussen}



\begin{document}
\begin{CJK}{UTF8}{min}

\begin{abstract}
Inspired by methods of N.~P.~Smart, we describe an algorithm to determine all Picard curves over $\Q$ with good reduction away from 3, up to $\Q$-isomorphism. A correspondence between the isomorphism classes of such curves and certain quintic binary forms possessing a rational linear factor is established. An exhaustive list of integral models is determined, and an application to a question of Ihara is discussed.
\end{abstract}

\maketitle

\section{Introduction}\label{sec:Intro}
Let $k_0$ be a field of characteristic other than $2$ or $3$. A Picard curve $\mathcal{C}/k_0$ is a smooth, projective, absolutely irreducible cyclic trigonal curve of genus $3$. Over any extension $k/k_0$ possessing a primitive cube root of unity, the required degree $3$ morphism $x \colon \mathcal{C} \to \P^1$ is Galois, and one always may recover from this map a smooth affine model of the form
\[ y^3 = F(x,1), \]
where $F(X,Z) \in k_0[X,Z]$ is a binary quartic form. Picard curves are always non-hyperelliptic, and provide the simplest examples of curves of gonality $3$. They have generated interest from geometric, arithmetic, and computational perspectives.  Picard curves and their moduli arise in connection with some Hilbert problems \cite{Holzapfel:1995}, and the monodromy of such moduli spaces have been studied in \cite{Achter-Pries:2007}.  Efficient algorithms have been found for arithmetic on the associated Jacobian varieties over finite fields \cite{Galbraith-Paulus-Smart:2002, Bauer:2004, Flon-Oyono:2004}.   Algorithms for counting points \cite{Bauer-Teske-Weng:2005} and statistics on point counts for Picard curves and the associated Jacobian varieties have also been studied \cite{Wood:2012, Bucur:2009}.

In the present article, we enumerate all $\Q$-isomorphism classes of Picard curves $\mathcal{C}/\Q$ with good reduction away from $3$. Because both the geometry and the arithmetic of such curves is exceptional, they have some applications to an open question of Ihara, which we discuss briefly at the end of the paper.

\subsection{Blueprint of Smart}
In \cite{Smart:1997}, N.~P.~Smart enumerated all curves $\mathcal{C}/\Q$, up to $\Q$-isomorphism, subject to the following constraints:
\begin{itemize}
\item $\mathcal{C}$ has genus $2$,
\item $\mathcal{C}$ has good reduction away from the prime $2$.
\end{itemize}
Because all such curves are hyperelliptic, they are naturally equipped with a degree $2$ morphism to $\P^1$. To carry out the exhaustive search, Smart first establishes a one-to-one correspondence
\[
\left\{ \begin{array}{c}
   \text{$\Q$-isomorphism classes} \\
   \text{of genus $2$ curves $\mathcal{C}/\Q$} \\
   \text{good away from $2$}
\end{array} \right\}_{/\text{twists}}
 \quad \stackrel{\cong}{\longleftrightarrow} \quad
\left\{ \begin{array}{c}
   \text{$\Z[\tfrac{1}{2}]$-equivalence classes} \\
   \text{of binary sextic forms} \\
   \text{good away from $2$}
\end{array} \right\}.
\]
(The definition of $\Z[\tfrac{1}{2}]$-equivalence is explained in \S\ref{sec:NotationsBackground}.) The correspondence is straightforward: Given a representative $F(X,Z)$ of a class from the set on the right, the associated curve $\mathcal{C}$ has an affine model of the form $y^2 = F(x,1)$. However, this association is not quite one-to-one; if $K/\Q$ is a quadratic extension unramified away from $2$, then the set on the right cannot distinguish between a curve $\mathcal{C}$ and its quadratic twist by $K$. The subscript ``twists'' indicates that the association is a bijection up to this family of twists; of course, recovering the distinct twists is trivial once one curve in the family is known, since there are only a finite number of extensions $K/\Q$ with the necessary properties.

Any class from the right hand set has a representative $F(X,Z)$ whose numerical invariants may be expressed in terms of solutions to an $S$-unit equation over the splitting field of $F$. By \cite[Lemma 3]{Birch-Merriman:1972}, this field is bounded in degree and may ramify only at $2$. Hence, there are only finitely many fields to consider, and each yields only finitely many solutions to the $S$-unit equation. Smart then gives the results of a complete search; first finding all possible $S$-unit equation solutions (thus all possible numerical invariants), and then determining at least one representative for each of the desired classes of sextic forms. Recovering the genus $2$ curves is now immediate from the correspondence. The process does not boil down to brute force, however. Smart employs several nontrivial algorithms to reduce the search to one of manageable size.

\subsection{Picard curves via Smart's approach}
As a hyperelliptic curve $\mathcal{C}/k$ is equipped with a degree $2$ morphism $\mathcal{C} \to \P^1$, one might suspect that the generalization of Smart's approach to the case of Picard curves will be immediate. That is, should one expect the association
\[ \left\{ \begin{array}{c} \text{$\Q$-isomorphism classes} \\ \text{of Picard
    curves $\mathcal{C}/\Q$} \\ \text{good away from $3$} \end{array} \right\}_{/\text{twists}}
\quad \longrightarrow \quad \left\{
\begin{array}{c} \text{$\Z[\tfrac{1}{3}]$-equivalence classes} \\
  \text{of binary quartic forms} \\ \text{good away from $3$} \end{array}
\right\} \]
also to be a one-to-one correspondence, up to twists? Unfortunately, no: the association of classes of forms to isomorphism classes of Picard curves is not even well-defined, as there are $\Z[\frac{1}{3}]$-equivalent quartic forms $F$, $G$ whose associated Picard curves
\[ \mathcal{C}_F \colon y^3 = F(x,1), \qquad \mathcal{C}_G \colon y^3 = G(x,1), \]
are not isomorphic over $\Q$. In \S\ref{sec:PicardCurves}, we introduce a finer notion of equivalence which will let us establish a correspondence between isomorphism classes of Picard curves and equivalence classes of forms.

Once this correspondence is established, the search for Picard curves with good reduction away from $3$ follows Smart's program in many important respects. However, the project diverged from Smart's work in the more delicate nature of the curves involved, which necessitated developing modified algorithms for determining all equivalence classes.  This paper also includes one theoretical improvement to Smart's approach: Lemma \ref{lemma:infinite place} makes the use of a $p$-adic LLL reduction unnecessary. We may avoid using the bounds on $p$-adic logarithmic forms established by Yu \cite{Yu:2007}, as well as the $p$-adic LLL algorithm used by Smart \cite[\S4]{Smart:1995}. The primary result of the paper is the following:
\begin{main_theorem*}
Up to $\Q$-isomorphism, there are exactly $63$ Picard curves $\mathcal{C}/\Q$ with good reduction away from $3$. Moreover, up to cubic twists, there are $21$ classes of such curves; each of these contains exactly $3$ $\Q$-isomorphism classes of Picard curves (corresponding to the twists by $\alpha = 1, 3, 9$).
\end{main_theorem*}
\subsection{Outline}
In \S\ref{sec:NotationsBackground}, we introduce some notation and review some background on
binary forms from Smart's methods. In \S\ref{sec:PicardCurves}, we introduce a finer notion of equivalence and establish the desired one-to-one correspondence between classes of Picard curves and classes of objects called quintic-linear pairs. In \S\ref{sec:QuarticSEquiv} and \S\ref{sec:QuarticS0Equiv}, we explain how to recover a representative from each equivalence class of quintic-linear pairs over $\Q$ with good reduction away from $3$. In \S\ref{sec:SolvingSUnit}, we describe the methods used to solve the $S$-unit equation. In \S\ref{sec:Conclusions}, we provide explicit models for the classes of Picard curves, and discuss further directions and applications of the
work.
\subsection*{Acknowledgments}
The authors are delighted to acknowledge the contribution of Michel B\"orner, at the Institute of Pure Mathematics, Ulm University. He brought to our attention that the list of isomorphism classes of curves provided in an earlier version of this paper was incomplete! B\"orner, after reading our previous version, had implemented this algorithm   independently, and contacted us when his own computation yielded classes we had missed. Upon review of our own implementation, we discovered an error that had led to the omissions. We are extremely grateful to him for the assistance. We emphasize that this computation has now been carried out independently by both B\"orner and the authors, and the results are in agreement.

In addition, the authors are grateful to Jeff Achter, Wai Kiu Chan, Norman Danner, David Grant, Rachel Pries, and John Voight for many helpful conversations during this project. We appreciate the detailed comments of the anonymous referees, which certainly improved the clarity of the paper and alerted the authors to an error in a previous version. This research was partially supported through Wesleyan University's Van Vleck Fund.

\section{Notations and Background}\label{sec:NotationsBackground}

\subsection{$S$-integers and $S$-units}
Let $K/\mathbb{Q}$ be a number field with $n_1$ real embeddings and $n_2$ pairs of conjugate complex embeddings (so $[K:\Q] = n_1 + 2n_2$), and let
\[S = \{  \mathfrak{p}_1, \dots, \mathfrak{p}_s, \mathfrak{p}_{s+1}, \dots, \mathfrak{p}_{s+n_1}, \mathfrak{p}_{s + n_1 + 1}, \dots, \mathfrak{p}_{s + n_1 + n_2} \}, \]
be a finite set of places of $K$ which includes all the infinite places of $K$ and exactly $s$ finite places. By convention, we enumerate first the finite places, then the real, then the complex infinite places. Let $\OK$ be the ring of integers of $K$, and let $\OS$ be the ring of $S$-integers in $K$. If $K = \Q$ and $S = \{p_1, \dots, p_r, \infty \}$, then the ring $\OS$ has the form
\[ \OS = S^{-1}\Z = \Z[ \tfrac{1}{p_1}, \cdots, \tfrac{1}{p_r}]. \]
For any $\OS$-ideal $\mathfrak{a}$, let $\mathfrak{a}^*$ denote an $\OK$-ideal, coprime to the finite places of $S$, such that $\mathfrak{a} = \mathfrak{a}^*\OS$. The \emph{$\OS$-norm} of $\mathfrak{a}$ is
\[ |\mathfrak{a}|_S := \mathbf{N}_{K/\Q}(\mathfrak{a}^*)^{1/[K:\Q]}. \]
For any $\alpha \in \OS$, we set $|\alpha|_S = |\alpha \OS|_S$. We let $\OS^\times$ be the group of $S$-units in $K$.

Let $S_\Q$ be a fixed, finite set of rational primes, together with $\infty$, the infinite place of $\Q$. Often the set $S$ will contain precisely those places of $K$ lying above a place of $S_\Q$. Even more explicitly, we usually have in mind $S_\Q = \{3, \infty \}$. We record the following fact for later use in solving $S$-unit equations.
\begin{lemma}\label{lemma:small_nf_disc_3}
Suppose $K$ is a number field with $[K:\Q] \leq 6$, for which the extension $K/\Q$ is unramified away from $\{3, \infty \}$. Then
\begin{enumerate}[(i)]
\item there is a unique, totally ramified prime $\mathfrak{p}$ in $\OK$ above $(3)$, and
\item the class number of $K$ is $1$.
\end{enumerate}
\end{lemma}
\begin{proof}
There are only finitely many fields $K$ which satisfy the hypotheses of the lemma; a complete list may be found from the number field database of Jones and Roberts \cite{Jones-Roberts:NF}. The verification of (i) and (ii) is immediate for each field by explicit computation.
\end{proof}
Table \ref{table:nf_disc_3} records all the number fields for which the lemma applies, and which either have degree $\leq 4$, or are the Galois closure of such a field. In particular, fields $K_0, K_1, K_2$ are Galois; $K_3$, $K_3'$, $K_3''$ are conjugate, and $L_3$ is the Galois closure of $K_3$. (There are $2$ additional sextic fields unramified away from $\{3, \infty \}$, but we will not need these in the sequel.)
\begin{table}[ht!]
\caption{Number fields of small degree unramified outside $\{3, \infty\}$}
\label{table:nf_disc_3}
{\renewcommand{\arraystretch}{1.1}
\begin{tabular}{lcr}
\toprule
Field & Degree & Minimal Polynomial \\
\midrule
$K_0$ & $1$ & $x - 1$ \\
$K_1$ & $2$ & $x^2 + x + 1$ \\
$K_2$ & $3$ & $x^3 - 3x + 1$ \\
$K_3$ & $3$ & $x^3 - 3$ \\
$K_3'$ & $3$ &  \\
$K_3''$ & $3$ & \\
$L_3$ & $6$ & $x^6 + 3$ \\
\bottomrule
\end{tabular}
}
\end{table}

\subsection{Binary forms over number fields}
Keep $K$ and $S$ as in the previous section. Let
\[ F(X,Z) = a_r X^r + \cdots + a_2 X^2 Z^{r-2} + a_1 XZ^{r-1} + a_0 Z^r \in \OS[X,Z] \]
be a homogeneous binary form of degree $r$ with coefficients in $\OS$. Then $F(X,Z)$ factors over some finite extension $L/K$ as
\[ F(X,Z) = \lambda \prod_{i=1}^{r}(\alpha_i X + \beta_i Z), \qquad \lambda \in K^\times, \alpha_i, \beta_i \in L. \]
It is useful to write this in the shorthand
\[ F(X,Z) = \lambda \prod_{i=1}^r \langle \mathbf{a}_i, \mathbf{X} \rangle, \]
where $\mathbf{X} = (X, Z)^T $ and $\mathbf{a}_i = (\alpha_i, \beta_i)^T$.

\subsubsection{Discriminant}
By scaling the $\alpha_i, \beta_i$ appropriately (and possibly extending $L$), we may assume $\lambda = 1$. Then the \emph{discriminant} of $F(X,Z)$ is the quantity
\[ D(F) = \prod_{1 \leq i < j \leq r} (\alpha_i \beta_j - \alpha_j \beta_i)^2 \in K. \]
It vanishes if and only if one linear factor of $F$ is a scalar multiple of another linear factor (i.e., $F(x,1)$ has a repeated root). We note that for any $\mu \neq 0$, $D(\mu F) = \mu^{2r-2}D(F)$.

Since the form $F$ has integral coefficients, we may consider its reduction modulo any prime $\mathfrak{p}$ of $\OK$. This gives a binary form $\overline{F}$ over the residue field $\OK/\mathfrak{p}\OK$. If $\deg F > 1$, we say $F$ has \emph{good reduction at $\mathfrak{p}$} if the reduced form $\overline{F}$ has no repeated linear factors up to scaling. Equivalently, $D(\overline{F}) \neq 0$ in the residue field, or what is the same, $\mathfrak{p} \nmid D(F)\OK$. We say $F$ has \emph{good reduction outside $S$} if $F$ has good reduction at every $\mathfrak{p} \not\in S$. This is equivalent to the condition $D(F) \in \OS^\times$. In case $F$ is linear (in which case the discriminant is always $1$), we say $F = \alpha X + \beta Z$ has good reduction at $\mathfrak{p}$ if $\mathfrak{p} \nmid (\alpha, \beta)\OK$. Thus, a linear form has good reduction outside $S$ if and only if $(\alpha, \beta)\OS = \OS$.

\subsubsection{$\OS$-equivalence}
Let $R$ be a commutative ring with $1$, and $F \in R[X,Z]$ a binary form of degree $r$. For a matrix
\[ U = \begin{pmatrix} a & b \\ c & d \end{pmatrix} \in GL_2(R), \]
we define $F_U(X,Z) = F(U\mathbf{X}) = F(aX + bZ, cX + dZ)$. We say two binary forms $F, G \in R[X,Z]$ are \textit{$R$-equivalent}, and write $F \sim_R G$, if there exist $U \in GL_2(R)$ and $\lambda \in R^\times$ such that $G = \lambda F_U$. This is an equivalence relation on binary forms of degree $r$. It is also the notion of equivalence of forms mentioned in the introduction when $R = \Z[\frac{1}{2}]$ or $R = \Z[\frac{1}{3}]$. The form $F_U$ has discriminant
\[ D(F_U) = (\det U)^{r(r-1)} D(F). \]
If $R = \OS$ and $D(F) \in \OS^\times$, then $D(G) \in \OS^\times$ for any $G \sim_{\OS} F$. Thus, the property of possessing good reduction outside $S$ is preserved under $\OS$-equivalence.

\subsubsection{Proper factorizations}
We recall the notion of field systems from \cite{Smart:1997}. Let us assume $K = \Q$ and $S = S_\Q$. Fix an algebraic closure $\overline{\Q}$ of $\Q$. Then a binary form ${F \in \OS[X,Z]}$ has a factorization over $\Q$ of the form
\[ F = \lambda F_1 \cdots F_m, \qquad \lambda \in \Q^\times, \quad F_j \in \Q[X,Z] \; \text{irreducible}. \]
For each $1 \leq j \leq m$, let $M_j \subseteq \overline{\Q}$ be a minimal extension of $\Q$ over which $F_j(x,1)$ admits a root, or $M_j = \Q$ if $F_j = Z$. The field $M_j$ is unique up to conjugation. The \emph{field system} of $F$ is the $m$-tuple $(M_1, \dots, M_m)$. It is well-defined up to conjugation of the $M_j$ and the indexing of the irreducible factors of $F$. Moreover, the field systems of two $\OS$-equivalent forms must be the same up to conjugation.

Let $M$ be the Galois closure of the compositum $M_1 \cdots M_m$. If $r_i = [M_i:\Q]$, then each $M_i$ admits $r_i$ embeddings $\psi_{ij} \colon M_i \hookrightarrow \overline{\Q}$. We extend the field system to a list of exactly $r$ fields $(M_1, \dots, M_m, M_{m+1}, \dots, M_r)$, by appending the conjugate fields $\psi_{ij}(M_i)$. Of course, we may have repeated fields in this list or even in the field system.  The field systems considered in this project are listed in Table \ref{table:field-systems}.

\begin{table}[!ht]
\centering
\caption{Relevant field systems for $r = 4$}
\label{table:field-systems}
{\renewcommand{\arraystretch}{1.05}
\begin{tabular}{ccccccc}
\toprule
$\mathscr{M}$ & $M_0$ & $M_1$ & $M_2$ & $M_3$ & & $M$ \\
\toprule
$(K_0, K_0, K_0, K_0)$ & $K_0$ & $K_0$ & $K_0$ & $K_0$ & & $K_0$ \\
\midrule
$(K_0, K_0, K_1)$ & $K_0$ & $K_0$ & $K_1$ & $K_1$ & & $K_1$ \\
\midrule
$(K_0, K_2)$ & $K_0$ & $K_2$ & $K_2$ & $K_2$ & & $K_2$ \\
\midrule
$(K_0, K_3)$ & $K_0$ & $K_3$ & $K_3'$ & $K_3''$ & & $L_3$ \\
\midrule
$(K_1, K_1)$ & $K_1$ & $K_1$ & $K_1$ & $K_1$ & & $K_1$ \\
\bottomrule
\end{tabular}}
\end{table}
Let $\Sigma_r$ denote the permutation group on $\{1, \dots, r\}$.
\begin{lemma}\label{lemma:indexing}
After rescaling $\lambda$, re-indexing the $\mathbf{a}_i$, and re-indexing the fields $M_i$ for $1 \leq i \leq r$, we may assume that the factorization $F = \lambda \prod \langle \mathbf{a}_i, \mathbf{X} \rangle$ satisfies $\mathbf{a}_i \in M_i^2$. Moreover, there exists an injective group homomorphism $\iota \colon \Gal(M/\Q) \to \Sigma_r$ such that for every $i$ with $1 \leq i \leq r$, $\mathbf{a}_i^\sigma = \mathbf{a}_{\iota(\sigma)(i)}$.
\end{lemma}
\begin{proof}
The proof follows by enumerating the various embeddings of the $M_i$ and then re-indexing carefully; the details are explained in \cite[pg.~273-274]{Smart:1997}.
\end{proof}
\begin{remark}
The homomorphism $\iota$ allows us to identify $\Gal(M/\Q)$ as a subgroup of $\Sigma_r$ in a way that respects the action on the $\mathbf{a}_i$. For the remainder of the paper, we will abuse notation and suppress the homomorphism $\iota$. Hence, we simply write $\mathbf{a}_{\sigma(i)}$ for $\mathbf{a}_{\iota(\sigma)(i)}$.
\end{remark}
A factorization of $F(X,Z)$ satisfying the conclusion of Lemma \ref{lemma:indexing} is called a \textit{proper factorization}. For each field $M_i$, let $S_i$ denote the set of places in $M_i$ which lie above a place of $S$, and let $T$ denote the set of places in $M$ which lie above a place of $S$. A proper factorization of $F$ is called \textit{$S$-proper} if
\begin{enumerate}[\quad (1)]
\item $D(F) \in \OS^\times$,
\item $\lambda = 1$,
\item $\mathbf{a}_i \in \mathscr{O}_{S_i}^2$ for each $i$,
\item the $\mathscr{O}_{S_i}$-ideal $(\mathbf{a}_i) = (\alpha_i, \beta_i)\mathscr{O}_{S_i}$ is $\mathscr{O}_{S_i}$.
\end{enumerate}
By \cite[Lemma 1]{Smart:1997}, each $\OS$-equivalence class of binary forms of degree $r$ with good reduction outside $S$ contains a representative $F$ which has an $S$-proper factorization.

\subsection{Arithmetic data of binary forms}
We keep $K = \Q$, $S = S_\Q$ in this section. Throughout, let $F(X,Z) \in \OS[X,Z]$ be a binary form of degree $r$ with $D(F) \in \OS^\times$, with a known $S$-proper factorization $F = \prod_i \langle \mathbf{a}_i, \mathbf{X} \rangle$. For any indices $i,j$ with $1 \leq i,j \leq r$, define
\[ \Delta_{i,j} = \det( \mathbf{a}_i \; \mathbf{a}_j) = (\alpha_i \beta_j - \alpha_j \beta_i). \]
We have $\Delta_{i,j} = -\Delta_{j,i}$ and $\Delta_{i,i} = 0$. Moreover, each $\Delta_{i,j} \in \mathscr{O}_T^\times$ if $i \neq j$, and $D(F) = \pm \prod_{i\neq j} \Delta_{i,j}$. We let $\bDelta(F) = (\Delta_{i,j}) \in M_{r \times r}(\OT)$. This is called the \emph{companion matrix} of the $S$-proper factorization.

For any index $i$ with $1 \leq i \leq r$, we also define
\[ \Omega_i = \prod_{k \neq i} \Delta_{i,k} \in \OT^\times. \]
In fact, we may be sure $\Omega_i \in \mathscr{O}_{S_i}^\times$: if $\sigma$ acts trivially on $M_i$, then it fixes $\mathbf{a}_i$, and so $\Delta_{i,j}^\sigma = \Delta_{i,\sigma(j)}$. Thus, $\Omega_i^\sigma = \prod_{k \neq i} \Delta_{i,k}^\sigma = \prod_{k \neq i} \Delta_{i,\sigma(k)} = \Omega_i$.

Let $\{ \rho_j \}_{j=1}^t$ be a $\Z$-basis for the torsion-free part of $\OT^\times$. Then there exists a root of unity $\zeta_i \in M$ and $a_{j,i} \in \Z$ such that
\begin{equation}\label{eq:Omega}
\Omega_i = \zeta_i \cdot \prod_j \rho_{j,i}^{a_{j,i}}.
\end{equation}
For a particular $S$-proper factorization, we have little control over the exponents $a_{j,i}$. However, up to $\OS$-equivalence, we are guaranteed small exponents for the values $\Omega_1, \dots, \Omega_m$.
\begin{lemma}\label{lemma:SmartLemma2}
Suppose $F'$ is a binary form of degree $r$ with $D(F') \in \OS^\times$. Then $F'$ is $\OS$-equivalent to a binary form $F$ with an $S$-proper factorization for which the exponents $a_{j,i}$ of \eqref{eq:Omega} simultaneously satisfy
\[ 0 \leq a_{j,i} < (r-2)(2r-2) \]
for every $1 \leq i \leq m$.
\end{lemma}
\begin{proof}
This is Lemma 2 of \cite{Smart:1997}.
\end{proof}

When $r \geq 4$, we also define cross ratios of $F(X,Z)$ as follows. For any choice of pairwise distinct indices $i,j,k,\ell$, the \emph{cross ratio} $[i,j,k,\ell]$ is the quantity
\[ [i, j, k, \ell] = \frac{\Delta_{i,j}\Delta_{k,\ell}}{\Delta_{i,k}\Delta_{j,\ell}} \in \mathscr{O}_T^\times. \]
Directly from the definition of $\Delta_{i,j}$, we see the cross ratios satisfy the identity
\begin{equation}\label{eq:cross ratio}
[i, j, k, \ell] + [k, j, i, \ell] = 1.
\end{equation}
Thus, in any $S$-proper factorization, the cross ratios are $T$-units which satisfy the equation $x + y = 1$. The classical result of Siegel \cite{Siegel:1929} implies each cross ratio has only finitely many possible values. Gy\H{o}ry provided the first effective bounds for the number of such solutions \cite{Gyory:1979}, using Baker's method \cite{Baker:1967}. Smart demonstrated algorithms for finding all solutions in certain cases.

Fix indices $i, j$ with $1 \leq i, j \leq r$, $i \neq j$. From \cite[pg.~276]{Smart:1997}, we have the equality
\begin{equation}
\label{delta equation}
\Delta_{i,j}^{(r-1)(r-2)}= (-1)^r \frac{ (\Omega_i \Omega_j)^{(r-1)} \prod^* [i, j, k, \ell] }{ \Omega_1 \cdots \Omega_r},
\end{equation}
where $\prod^*$ is taken over the pairs $(k,\ell)$ for which $k \neq i, j$ and $\ell \neq i, j, k$. From this formula, the values $\Delta_{i,j}$ may be recovered from the values of the $\Omega_i$ and the cross ratios. By Lemma \ref{lemma:SmartLemma2} and the finiteness of $T$-unit equation solutions, it follows that, up to $\OS$-equivalence, there are only finitely many choices for the companion matrix $\bDelta(F)$ of a binary form $F$ of degree $r$ satisfying $D(F) \in \OS^\times$.

\begin{remark}
Recall that we have identified $\Gal(M/\Q)$ with a subgroup of $\Sigma_r$ in such a way that $\mathbf{a}_i^\sigma = \mathbf{a}_{\sigma(i)}$ for every $1 \leq i \leq r$ and every $\sigma \in \Gal(M/\Q)$. As a consequence, the quantities $\Delta_{i,j}$, $\Omega_i$, and $[i,j,k,\ell]$ respect analogous formulas:
\[ \Delta_{i,j}^\sigma = \Delta_{\sigma(i), \sigma(j)}, \quad \Omega_i^\sigma = \Omega_{\sigma(i)}, \quad [i,j,k,\ell]^\sigma = [\sigma(i), \sigma(j), \sigma(k), \sigma(\ell)]. \]
\end{remark}

\section{Isomorphism classes of Picard curves}\label{sec:PicardCurves}
Again let $k_0$ be a field with characteristic not $2$ or $3$, and let $\mathcal{C}_0/k_0$ be a Picard curve. Let $k/k_0$ be an extension containing a primitive cube root of unity, and set $\mathcal{C} = \mathcal{C}_0 \times_{k_0} k$. By the Riemann-Hurwitz formula, the trigonal covering $\mathcal{C} \to \P^1$ branches at exactly $5$ points, and by Kummer theory, the extension $k(\mathcal{C})/k(x)$ corresponding to the trigonal map must be generated by the cube root of some element $f \in k(x)$. Thus, we may choose $y \in k(\mathcal{C})$ such that $y^3 = f(x)$ and $k(\mathcal{C}) = k(x,y)$. We may scale the rational function $f(x)$ by any cube in $k(x)$, and so without loss of generality we may take $f$ in the form
\[ f(x) = \lambda \prod_{i=1}^N (x - \alpha_i)^{a_i}, \qquad \lambda \in k^\times, \quad 1 \leq a_i \leq 2, \quad 4 \leq N \leq 5. \]
Since $f$ is a polynomial, $\infty$ is a branch point if and only if $3 \nmid \deg f = \sum_i a_i$. Replacing $f$ by $\prod_i (x - \alpha_i)^3 / f$ if necessary, we may assume $\deg f \leq 6$. Thus, an affine model for $\mathcal{C}$ has one of the forms
\begin{subequations}
\begin{align}
y^3 & = \lambda(x - \alpha_1)(x - \alpha_2)(x - \alpha_3)(x - \alpha_4), \label{eq:aff_quartic} \\
y^3 & = \lambda(x - \alpha_1)^2(x - \alpha_2)(x - \alpha_3)(x - \alpha_4), \\
y^3 & = \lambda(x - \alpha_1)^2(x - \alpha_2)(x - \alpha_3)(x - \alpha_4)(x - \alpha_5),
\end{align}
\end{subequations}
with $\alpha_i \in \overline{k}$. As the trigonal cover is defined over $k_0$, it follows that the branch locus is $G_{k_0}$-stable. Consequently, the monic polynomial $\lambda^{-1}f$ in fact lies in $k_0[x]$, and due to the appearance of the repeated root in the latter two forms, we are guaranteed that at least one branch point of $\mathcal{C}$ is defined over $k_0$.

Suppose that $\phi$ is an automorphism of $k(x)$ (i.e., a fractional linear transformation of $\P^1_{k}$), and set $g = \phi(f)$. Then $\phi$ extends to an isomorphism
\[ k(x, \sqrt[3]{f}) \longrightarrow k(x, \sqrt[3]{g}), \]
giving an isomorphism between the affine models $y^3 = f$ and $y^3 = g$. As the group $\Aut(\P^1_{k_0}) \cong PGL_2(k_0)$ is $3$-transitive, we may always choose $\phi$ so that the $k_0$-rational branch point for the affine model $y^3 = g$ occurs at $\infty$. Consequently, we may assume that $\mathcal{C}$ has a projective equation of the form
\[ \lambda Y^3 Z = Q(X,Z), \qquad Q \in k_0[X,Z] \]
where $Q$ is a squarefree quartic form. In order to classify Picard curves with prescribed reduction up to isomorphism, we must understand how an isomorphism of (integral separated models of) Picard curves translates into equivalence of binary forms.

\subsection{Picard curves over number fields}
Let $K$ be a number field. By \cite{Estrada-Sarlabous:1995}, a Picard curve $\mathcal{C}/K$ always has a smooth projective model (called a \emph{separated form}) with equation $Y^3Z = F(X,Z)$, where
\[ F(X,Z) = a_4 X^4 + a_3 X^3 Z + a_2 X^2 Z^2 + a_1 X Z^3 + a_0 Z^4, \qquad a_4 \neq 0, \; a_i \in K, \]
is a quartic binary form in $K[X,Z]$. There is a corresponding affine model,
\[ y^3 = F(x,1), \]
which omits the single point $(0:1:0)$ at infinity.

The curve $\mathcal{C}$ is said to have \emph{good reduction} at the prime $\mathfrak{p}$ if the special fiber $\mathcal{C}_\mathfrak{p}$ of $\mathcal{C}$ is nondegenerate and smooth. Let $S$ be a finite set of primes in $K$, and suppose $\mathcal{C}$ has good reduction at all primes $\mathfrak{p} \not\in S$. A \textit{normal form} for $\mathcal{C}$ is an $S$-integral model
\[ Y^3Z - F(X,Z) = 0, \]
satisfying $F \in \OS[X, Z]$, $D(F) \in \OS^\times$ and $a_4 = 1$.\footnote{We drop the condition $a_3 = 0$ from \cite{Estrada-Sarlabous:1995}.} By \cite[Lemma 7.5]{Estrada-Sarlabous:1995}, every Picard curve over $K$ with good reduction away from $S$ admits a normal form.

Unfortunately, even if $F, G \in \OS[X,Z]$ are $\OS$-equivalent binary forms with $D(F), D(G) \in \OS^\times$, it does not follow that the Picard curves
\[ \mathcal{C}_F \colon Y^3Z - F(X,Z) = 0, \qquad \mathcal{C}_G \colon Y^3Z - G(X,Z) = 0, \]
are isomorphic over $K$, even when $K$ possesses a cube root of unity. This may seem counter-intuitive, so we explain in a bit more detail. If $G = F_U$ for some matrix $U = ( \begin{smallmatrix} a & b \\ c & d \end{smallmatrix} ) \in GL_2(\OS)$, then set $f(x) = F(x,1)$, $g(x) = G(x,1)$. We have $f, g \in K(x)$, but the fractional linear transformation
\[ u \colon x \mapsto \frac{ax + b}{cx + d}, \qquad u \in \Aut( K(x) ), \]
corresponding to $U$ satisfies $g = (cx + d)^4 \cdot (f \circ u)$. Thus, we cannot expect the Kummer extensions of $K(x)$ generated by the cube roots of $f$ and $g$ to be isomorphic; but these are the function fields of $\mathcal{C}_F$ and $\mathcal{C}_G$, respectively!
\begin{example}
For an explicit example, the reader may take $F = X^4 - XZ^3$ and $U = (\begin{smallmatrix} 1 & 1 \\ 2 & 1 \end{smallmatrix})$. With $G = F_U$ we have $D(F) = D(G) = -27$, but $\mathcal{C}_F \not\cong \mathcal{C}_G$.
\end{example}
Notice, however, that if $c = 0$ we do obtain an isomorphism (up to a cubic twist by $d$). This motivates the next definition.

Let $R$ be a commutative ring with $1$, and let $F, G \in R[X,Z]$ be binary forms of degree $r$. We say $F$ and $G$ are \emph{$R^0$-equivalent}, and write $F \sim_{R^0} G$, if there exists $\lambda \in R^\times$ and
\[ U = \begin{pmatrix} a & b \\ 0 & d \end{pmatrix} \in GL_2(R) \]
such that $G = \lambda F_U$. Of course, $R^0$-equivalence implies $R$-equivalence, but the reverse is not generally true.
\subsection{Tschirnhaus transformations}
We wish to describe the automorphisms of $\P^2$ which induce isomorphisms of Picard curves. A \emph{Tschirnhaus transformation} is an automorphism of $\P^2_K$ given by an element $\tau \in PGL_3(K)$ of the form
\[ \begin{pmatrix} X \\ Y \\ Z \end{pmatrix} = \tau \begin{pmatrix} X' \\ Y' \\ Z' \end{pmatrix}, \quad \tau = \begin{pmatrix} u^3 & 0 & r \\ 0 & u^4 & 0 \\ 0 & 0 & 1 \end{pmatrix}, \qquad u \in K^\times, r \in K. \]
Fix a finite set of primes $S_\Q$ in $\Q$, and suppose $F, G \in \mathscr{O}_{S_\Q}[X,Z]$ are quartic forms for which $\mathcal{C}_F$ and $\mathcal{C}_G$ are isomorphic over $\Q$. Then the isomorphism may be given as a Tschirnhaus transformation $\tau$, with $u \in \mathscr{O}_{S_\Q}^\times$, $r \in \mathscr{O}_{S_\Q}$. (This follows from Lemma 7.5 and Remark 7.7 of \cite{Estrada-Sarlabous:1995}.) As a consequence, we have $G = F_U$, where
\[ U = \begin{pmatrix*}[l] 1 & u^{-3}r \\ 0 & u^{-3} \end{pmatrix*}. \]
Thus, $F$ and $G$ are $\mathscr{O}_{S_\Q}^0$-equivalent. Moreover, the converse statement is true, up to a cubic twist.
\begin{lemma}\label{c=0 means isomorphic}
Let $S_\Q$ be a finite set of places of $\Q$, with $\infty \in S$. Then there is a finite subset $A \subseteq \Q$ with the following property. Given any pair $F, G \in \mathscr{O}_{S_\Q}[X,Z]$ of $\mathscr{O}_{S_\Q}^0$-equivalent quartic binary forms, there exists $\alpha \in A$ such that the curve $\mathcal{C}_F$ is isomorphic over $\Q$ to $\mathcal{C}_{\alpha G}$.
\end{lemma}
\begin{proof}
Suppose $S_\Q = \{\infty, p_1, \dots, p_s \}$. Set
\[ A := \left\{ \pm \prod_{i=1}^s p_i^{a_i} : 0 \leq a_i \leq 2 \right\}. \]
As $F$, $G$ are $\mathscr{O}_{S_\Q}^0$-equivalent, there exists $\lambda \in \mathscr{O}_{S_\Q}^\times$ and
\[ U = \begin{pmatrix} a & b \\ 0 & d \end{pmatrix} \in GL_2(\mathscr{O}_{S_\Q}) \]
such that $G = \lambda F_U$. Note $a, d \in \mathscr{O}_{S_\Q}^\times$. Consequently, $\lambda d \in \mathscr{O}_{S_\Q}^\times$ also. Thus, there exists $\alpha \in A$ such that $\alpha \lambda d \in \Q^{\times 3}$. Let $\xi$ denote the rational cube root of $\alpha \lambda d$. We define
\[ \phi = \begin{pmatrix} a & 0 & b \\ 0 & \xi^{-1} & 0 \\ 0 & 0 & d \end{pmatrix} \in PGL_3(\Q). \]
The transformation $\phi$ gives an automorphism of $\P^2$; we need only check that it satisfies $\phi(\mathcal{C}_{\alpha G}) = \mathcal{C}_F$. If $P = (x : y : z) \in \mathcal{C}_{\alpha G}$, then
\[ \phi(P) = (ax + bz : \xi^{-1}y : dz). \]
Keeping in mind that $P$ satisfies $y^3 z - \alpha \lambda F_U(x, z) = 0$, we verify that $\phi(P)$ lies on $\mathcal{C}_F$. Let $\Psi = Y^3Z - F(X,Z) \in \Q[X,Y,Z]$. Then $\phi(P) \in \mathcal{C}_F$ (by definition) if $\Psi(\phi(P)) = 0$. As
\[ \begin{split}
\Psi(\phi(P)) & = (\xi^{-1}y)^3 \cdot (dz) - F(ax + bz, dz) \\
& = \frac{y^3 dz}{\alpha \lambda d} - \frac{ \alpha \lambda}{\alpha \lambda} F_U(x,z) \\
& = (\alpha \lambda)^{-1} (y^3 - \alpha \lambda F_U(x, z)) = 0,
\end{split} \]
we see $\phi(\mathcal{C}_{\alpha G}) \subseteq \mathcal{C}_F$.

Similarly, we verify $\phi^{-1}(\mathcal{C}_F) \subseteq \mathcal{C}_{\alpha G}$. First, any point $Q = (x_1 : y_1 : z_1) \in \P^2$ satisfies
\[ \phi^{-1}(Q) = \left( \frac{dx_1 - bz_1}{ad} : \xi y_1 : \frac{z_1}{d} \right). \]
Presuming $Q \in \mathcal{C}_F$, it is routine to show $\phi^{-1}(Q) \in \mathcal{C}_{\alpha G}$. Thus, the restriction of $\phi$ to $\mathcal{C}_{\alpha G}$ is an explicit isomorphism $\mathcal{C}_{\alpha G} \to \mathcal{C}_F$.
\end{proof}

\begin{corollary}\label{cor:PicardCorrespondence}
Up to the cubic twists by $\alpha \in A$, we have a one-to-one correspondence
\[ \mathfrak{PC} := \left\{ \begin{array}{c} \text{$\Q$-isomorphism classes} \\ \text{of Picard
    curves $\mathcal{C}/\Q$} \\ \text{good away from $\SQ$} \end{array} \right\}_{/\text{twists}}
\quad \hspace*{-5ex} \stackrel{\cong}{\longleftrightarrow} \quad \left\{
\begin{array}{c} \text{$\OSQ^0$-equivalence classes} \\
  \text{of binary quartic forms} \\ \text{good away from $\SQ$} \end{array}
\right\}. \]
\end{corollary}
\begin{remark}
In the case $S_\Q = \{3, \infty \}$, we may take $A = \{1, 3, 9 \}$. The distinct twists in this case will be isomorphic over the splitting field of $x^3 - 3$.
\end{remark}

\subsection{Quintic-Linear Pairs}\label{quintic-linear pairs}

Unfortunately, a new problem arises: Smart's methods to determine a representative from each $\OS$-equivalence class of binary forms are not guaranteed to produce one representative from each $\OS^0$-equivalence class! In this section, we demonstrate how to compute $\OS^0$-equivalence classes of quartics by rephrasing the problem in terms of $\OS$-equivalence of quintic forms. In brief, we construct a finite-to-one correspondence between the $\OS^0$-equivalence classes of quartic forms with good reduction away from $S$, and a certain subset of the set of $\OS$-equivalence classes of quintic forms with good reduction away from $S$. This latter set can be computed by Smart's methods.

Let $\mathfrak{B}_r$ denote the set of $\OS$-equivalence classes of binary forms of degree $r$ defined over $K$ with good reduction outside $S$, and let $\mathfrak{B}_r^0$ denote the set of $\OS^0$-equivalence classes of such forms. Define the following subset of $\mathfrak{B}_r$:
\[ \mathfrak{B}_{r,1} := \{ [F] \in \mathfrak{B}_r : \text{there exists } L \in K[X,Z], L \mid F, \deg L = 1 \}. \]

Suppose $Q, L \in \OS[X,Z]$ are square-free binary forms with $\deg Q = 5$ and $\deg L = 1$. If $L \mid Q$, we call $(Q, L)$ a \emph{quintic-linear pair}. Two such pairs $(Q,L)$ and $(Q',L')$ are \emph{$\OS$-equivalent} if there exists $\lambda \in \OS^\times$ and $U \in GL_2(\OS)$ such that $Q' = \lambda Q_U$ and $L' = L_U$. This defines an equivalence relation on the set of quintic-linear pairs with good reduction away from $S$. Let $\mathfrak{QL}$ denote the set of $\OS$-equivalence classes of such pairs.
\begin{remark}
Any particular linear factor $L$ of $Q$ is unique only up to $\OS$-unit. However, if $\mu \in \OS^\times$, notice $(Q,L)$ and $(Q, \mu L)$ are $\OS$-equivalent by the choice $\lambda = \mu^{-5}$ and $U = (\begin{smallmatrix} \mu & 0 \\ 0 & \mu \end{smallmatrix})$.
\end{remark}
Notice that the forgetful map $\Phi \colon \mathfrak{QL} \twoheadrightarrow \mathfrak{B}_{5,1}$ defined by $[(Q,L)] \mapsto [Q]$ is surjective and finite-to-one. as every fiber of $\Phi$ has at most $5$ classes. Also, given a class $[Q] \in \mathfrak{B}_{5,1}$, computing the classes in the fiber over $[Q]$ is equivalent to factoring $Q$ over $K$. In the remainder of this section, we will construct a convenient set of representatives $\mathfrak{Z}_{5,1}$ for the set $\mathfrak{QL}$, and demonstrate a bijection between $\mathfrak{B}^0_4$ and $\mathfrak{QL}$. The following diagram may help the reader:
\begin{equation*}
\begin{tikzcd}
\mathfrak{PC} \arrow{r}{\cong} & \mathfrak{B}_4^0 \arrow{r}{\cong} & \mathfrak{Z}_{5,1} \arrow{r}{\cong} & \mathfrak{QL} \arrow{d}[two heads]{\Phi} & {} \\
& & & \mathfrak{B}_{5,1} \arrow{r}{\subseteq} & \mathfrak{B}_5
\end{tikzcd}
\end{equation*}
The leftmost bijection is simply Corollary \ref{cor:PicardCorrespondence}, and the bottom right inclusion is obvious. The next proposition will establish the remaining bijections.
\begin{proposition}\label{QL a and b}
\begin{enumerate}[(a)]
\item Every class in $\mathfrak{QL}$ contains a representative of the form $(ZG, Z)$ for some quartic form $G \in \OS[X,Z]$.
\item Let $F, G \in \OS[X,Z]$ satisfy $D(F), D(G) \in \OS^\times$. Then $F$ and $G$ are $\OS^0$-equivalent quartics if and only if $(ZF, Z)$ and $(ZG, Z)$ are $\OS$-equivalent quintic-linear pairs.
\end{enumerate}
\end{proposition}
\begin{proof}[Proof of (a)]
Suppose $(Q, L)$ is a quintic-linear pair and $Q$ has good reduction outside $S$. By scaling $L$ by an appropriate $\OS$-unit as necessary, we may assume $L$ is an $S$-proper linear form over $K$. Write $L(X,Z) = \alpha X + \beta Z$ with $\alpha, \beta \in \OS$. By definition, $(\alpha, \beta)\OS$ is the unit ideal in $\OS$, and so there exist $a, b \in \OS$ such that $a \alpha + b \beta = 1$. Setting
\[ U = \begin{pmatrix*} \beta & a \\ -\alpha & b \end{pmatrix*} \in GL_2(\OS), \]
we have $L_U = Z$. Consequently, $(Q,L)$ is $\OS$-equivalent to $(Q_U, Z)$, and $Z \mid Q_U$ as required.
\end{proof}
\begin{proof}[Proof of (b)]
First, suppose that $F$ and $G$ are $\OS^0$-equivalent quartic forms. Then there exists $\lambda \in \OS^\times$ and an upper-triangular matrix $U = (\begin{smallmatrix} a & b \\ 0 & d \end{smallmatrix}) \in GL_2(\OS)$ such that $G = \lambda F_U$. Clearly $d \in \OS^\times$, and notice that $Z_{d^{-1} U} = Z$. It is simple to verify
\[ ZG = (\lambda d^4) (ZF)_{d^{-1} U}, \]
and so $(ZF,Z)$ and $(ZG,Z)$ are $\OS$-equivalent quintic-linear pairs.

Now, suppose $(ZF,Z)$ and $(ZG,Z)$ are $\OS$-equivalent quintic-linear pairs with good reduction outside $S$, where $F$ and $G$ are quartic forms with good reduction outside $S$. Thus, there exist $\lambda \in \OS^\times$ and $U = (\begin{smallmatrix} a & b \\ c & d \end{smallmatrix} ) \in GL_2(\OS)$ such that
\[ ZG = \lambda (ZF)_U, \qquad Z = Z_U. \]
However, $Z_U = cX + dZ$, which implies $c = 0$ and $d = 1$. Since $(ZF)_U = Z_U F_U$, it follows also that $G = \lambda F_U$. Thus, $F$ and $G$ are $\OS^0$-equivalent.
\end{proof}
\begin{remark}
By part (a), we may select a representative of the form $(ZF, Z)$ from each class in $\mathfrak{QL}$; we construct the set $\mathfrak{Z}_{5,1}$ by selecting one such representative from each class. By definition, the sets $\mathfrak{QL}$ and $\mathfrak{Z}_{5,1}$ are in bijection, and by part (b), $\mathfrak{Z}_{5,1}$ is in bijection with $\mathfrak{B}^0_4$.  So the task of computing all Picard curves defined over $\mathbb{Q}$ with good reduction outside of $S$ is equivalent to finding $\mathfrak{Z}_{5,1}$ for each potential field system.  When $S=\{3,\infty\}$, each $F$ with $(ZF,Z)\in \mathfrak{Z}_{5,1}$ gives rise to the $3$ Picard curves $Y^3Z=\alpha F$ for $\alpha \in\{1,3,9\}$.
\end{remark}

\section{Quartic Forms up to $\OS$-Equivalence}\label{sec:QuarticSEquiv}

Notice that there cannot be an irreducible quintic form $F \in \Q[X,Z]$ with good reduction outside $S_\Q = \{3, \infty \}$, because there are no quintic fields satisfying Lemma \ref{lemma:small_nf_disc_3}. Therefore, following \cite[\S7]{Smart:1997}, we may build representatives of all $\OS$-equivalence classes of quintic forms inductively by first constructing representatives of all $\OS$-equivalence classes of quartic forms, then combining with linear forms.  The field systems we are concerned with are listed in Table \ref{table:field-systems}. Let $\mathscr{M}_0=(M_0,\dots,M_{m-1})$ be a field system with $M_i$ unramified away from $\{3, \infty \}$ for all $i$ and $ \sum_i [M_i : \mathbb{Q}] = 4$. For technical reasons, in this section we order the fields in the field system so that if $M_i=\mathbb{Q}$ for a unique $i$, then $i\neq 0$ (see the discussion after Lemma \ref{lemma:SmartLemma5}). Let $M$ be the Galois closure of the compositum of the fields $M_i$, let $S=S_{\mathbb{Q}}=\{3,\infty\}$, and let $T$ be the set of places of $M$ lying above the places in $S$.

We let $\brho = (\rho_0, \rho_1, \dots, \rho_t)$ be a list of generators for $\OT^\times$, chosen so that $\rho_0$ generates the torsion part, and $\{\rho_1, \dots, \rho_t\}$ is a $\Z$-basis for the free part. For any vector $\mathbf{a} = (a_0, a_1, \dots, a_t) \in \Z^{t+1}$, we use the shorthand
\[ \brho^\mathbf{a} := \prod_{j=0}^t \rho_j^{a_j}. \]
If $\varepsilon \in \OT^\times$ has the form $\varepsilon = \brho^\mathbf{a}$, then we call $\mathbf{a}$ the \emph{exponent vector} of $\varepsilon$. Since multiplication of $T$-units corresponds to the addition of exponent vectors, the following algorithm may be carried out entirely inside the lattice $\Z^{t+1}$, after step (1).

The construction of a set of quartic forms containing a representative of each $\mathscr{O}_S$-equivalence class of forms good away from $S$ proceeds most naturally by considering each possible field system in turn. So we introduce the following notation:
\[ \mathfrak{B}_r(\mathscr{M}_0) := \{ [F] \in \mathfrak{B}_r : F \text{ has field system } \mathscr{M}_0 \}. \]
The following algorithm computes a set of representatives, $\mathfrak{F}_4(\mathscr{M}_0)$, for the set $\mathfrak{B}_4(\mathscr{M}_0)$. There are 4 phases to constructing this set for each field system $\mathscr{M}_0$:

\begin{enumerate}[\bf 1.]
\item Determine a complete set of solutions $\mathfrak{T}(M)$ to the $T$-unit equation
\[ \tau_0 + \tau_1 = 1, \qquad \tau_i \in \OT^\times. \]
This is the most computationally expensive step, and is explained in detail in \S\ref{sec:SolvingSUnit}. For ease of computation in the later steps, we record a dictionary of exponent vectors. Because $\tau_0, \tau_1$ are related additively, there is no more efficient way to recover the exponent vector of $\tau_1$ given the exponent vector of $\tau_0$. The finite set $\mathfrak{T}(M)$ gives all possible values for a cross ratio appearing in \eqref{delta equation}.

\item Construct a finite collection $\Omega(\mathscr{M}_0)$ of possible lists $(\Omega_1, \dots, \Omega_r)$ of $T$-units which could arise from an $S$-proper factorization of a binary form with field system $\mathscr{M}_0$. It is only required to find such a list for one representative of each equivalence class. Thus, by Lemma \ref{lemma:SmartLemma2}, we may assume the exponents $a_{j,i}$ in \eqref{eq:Omega} are bounded. Moreover, there are Galois constraints on the $\Omega_i$: for any $\sigma \in \Gal(M/\Q)$, $\Omega_i^\sigma = \Omega_{\sigma(i)}$. This imposes various congruence conditions on the exponent vectors of the $\Omega_i$. As a practical matter for computation, this allows the collection $\Omega(\mathscr{M}_0)$ to be implemented as an iterator, rather than by a large and explicitly-computed list of lists.

\item Compute a finite set $\mathfrak{D}(\mathscr{M}_0)$ of matrices $\bDelta \in M_{r \times r}(\OT)$, which may possibly be the companion matrix to an $S$-proper factorization with field system $\mathscr{M}_0$. This proceeds as follows:
\begin{enumerate}[(a)]
\item Loop on $(\Omega_1, \dots, \Omega_r) \in \Omega(\mathscr{M}_0)$.
\item Loop on all possible choices for $[i,j,k,\ell] \in \mathfrak{T}(M)$. These are constrained by Galois compatibility and \eqref{eq:cross ratio}.
\item Check: Is the right hand side of \eqref{delta equation} a perfect $(r-1)(r-2)$-th power for all $1 \leq i < j \leq r$?
\begin{itemize}
\item If NO, iterate the loop in (b).
\item If YES, compute all possible $\bDelta$ which satisfy \eqref{delta equation}, and add them to the collection $\mathfrak{D}(\mathscr{M}_0)$.
\end{itemize}
\end{enumerate}
In practice, we may be more selective by also checking the Galois compatibility conditions on the entries of each $\bDelta$. Even still, there is no guarantee that any given $\bDelta \in \mathfrak{D}(\mathscr{M}_0)$ actually corresponds to an $S$-proper factorization.

\item[\bf 4.] Compute a set $\mathfrak{F}_4(\mathscr{M}_0)$ of binary forms with good reduction away from $S$, and with field system $\mathscr{M}_0$. The set will contain at least one representative from each $\OS$-equivalence class in $\mathfrak{B}_4(\mathscr{M}_0)$. The procedure is carried out as follows:
\begin{enumerate}[(a)]
\item Loop over $\bDelta \in \mathfrak{D}(\mathscr{M}_0)$.
\item Compute $\beta = \det B$ and and all possible $\Lambda$ from Lemma \ref{lemma:SmartLemma5} below.
\item Loop over $B' \in \mathscr{U}_\beta$, the finite computable set from Lemma \ref{lemma:UB} below.\footnote{But see the Remark at the end of this section.}  From $A' = B' \Lambda$ and $\bDelta$, construct the form $G = \prod \langle \mathbf{a}'_i, \mathbf{X} \rangle$.
\item Check whether $D(G) \in \OS^\times$. If so, add $G$ to the set $\mathfrak{F}_4(\mathscr{M}_0)$.
\end{enumerate}
\end{enumerate}

\begin{remark}
Step 4 depends on two lemmas from \cite{Smart:1997}, which for convenience we include here. We say two matrices $E, E' \in M_{2 \times 2}(\OS)$ are equivalent if there exists $U \in GL_2(\OS)$ such that $E = UE'$. For any $\delta \in \OS$, $\delta \neq 0$, define
\[ \mathscr{U}_\delta := \left\{ \begin{pmatrix} \theta & \psi \\ 0 & \phi \end{pmatrix} \in M_{2 \times 2}(\Z) : 0 < \theta \leq |\delta|_S, \theta \phi = |\delta|_S, 0 \leq \psi \leq \phi-1 \right\}. \]
\end{remark}
\begin{lemma}\label{lemma:UB}
Suppose $E \in M_{2 \times 2}(\OS)$, and let $\delta = \det(E)$. Then $E$ is equivalent to a matrix $E' \in \mathscr{U}_\delta$.
\end{lemma}
This is proven in \cite[Lemma 4]{Smart:1997}.
\begin{lemma}\label{lemma:SmartLemma5}
Suppose $F \in K[X, Z]$ has an $S$-proper factorization
\[ F = \prod_{i=0}^{r-1} \langle \mathbf{a}_i, \mathbf{X} \rangle, \qquad \mathbf{a}_i \in \OT^2, \]
with companion matrix $\bDelta$. Let $A = (\mathbf{a}_0 \; \mathbf{a}_1) \in GL_2(\OT)$. Then there exists $\Lambda \in GL_2(\OT)$ and $B = (\mathbf{b}_0 \; \mathbf{b}_1) \in GL_2(\OS)$ such that $A = B \Lambda$. Moreover, both $\Lambda$ and $\beta = \det B$ may be determined from the matrix $\bDelta$ alone.
\end{lemma}
This is proven in \cite[Lemma 5]{Smart:1997}.  The proof is constructive. In fact, the matrix $\Lambda$ may be computed in a routine manner from $\bDelta$ and an integral basis for $\mathscr{O}_{M_0}$. However, this step involves a search for a pair of linearly independent $M_0$-linear combinations of $\mathbf{a}_0$ and $\mathbf{a}_1$, which is impossible if $M_0 = \Q$ and $M_1 \neq \Q$; so we reorder the fields in $\mathscr{M}_0$ as necessary to avoid this. Once $\Lambda$ is determined, we have
\[ \beta = \frac{ \Delta_{0,1} }{\det \Lambda}. \]
See the proof in \cite{Smart:1997} for further details.

We now have the following observation. Suppose $F$ has an $S$-proper factorization with companion matrix $\bDelta$. The vectors $\mathbf{a}_i$ appearing in the $S$-proper factorization of $F$ satisfy the relations
\begin{equation}\label{eq:a-vectors}
\mathbf{a}_i = \frac{ \Delta_{1,i} }{ \Delta_{1,0} } \mathbf{a}_0 + \frac{ \Delta_{i,0} }{ \Delta_{1,0} } \mathbf{a}_1, \qquad 0 \leq i < r.
\end{equation}
Thus, $F$ may be reconstructed trivially if both $A$ and $\bDelta$ are known; but, alas, $A$ may not be known! However, we have the relation $A = B \Lambda$, and by Lemma \ref{lemma:UB}, there exists $B' \in \mathscr{U}_\beta$ with $B' = UB$ for some $U \in GL_2(\OS)$. Let $A' = B'\Lambda$, and note that $A' = UA$. For $0 \leq i < r$, set $\mathbf{a}_i' = U\mathbf{a}_i$. The relations \eqref{eq:a-vectors} imply
\[ \mathbf{a}'_i = \frac{ \Delta_{1,i} }{ \Delta_{1,0} } \mathbf{a}'_0 + \frac{ \Delta_{i,0} }{ \Delta_{1,0} } \mathbf{a}'_1, \qquad 0 \leq i < r. \]
Thus, we may recover the individual factors in the binary form $G = \prod \langle \mathbf{a}'_i, \mathbf{X} \rangle$ for any choice of $B' \in \mathscr{U}_\beta$. Necessarily, $G$ is $\OS$-equivalent to the form $F$ whose companion matrix is $\bDelta$.

At the end of this process, we have obtained for each $\mathscr{M}_0$ a finite set $\mathfrak{F}_4(\mathscr{M}_0)$ with the following properties. Every form in $\mathfrak{F}_4(\mathscr{M}_0)$ has field system $\mathscr{M}_0$ and good reduction outside $S$, and every $\OS$-equivalence class of binary forms with field system $\mathscr{M}_0$ and good reduction outside $S$ has at least one representative in the set $\mathfrak{F}_4(\mathscr{M}_0)$. We may screen for redundancies among the representatives using the methods from \cite[\S6]{Smart:1997}.

\begin{remark}
\emph{Improvement in Step 4(c)}. The process of iterating through all matrices $B' \in \mathscr{U}_\beta$ in Step 4(c) has the potential to be a serious bottleneck in the computation. For example, when running this computation for the field system $(K_0, K_3)$, there are $62208$ distinct companion matrices to check, and for each such companion matrix, the set of matrices $\mathscr{U}_\beta$ was often larger than $150000$ (corresponding to $|\beta|_S = 101194$ in the most extreme case). In practice, however, we were able to reduce the running time with the following trick. Recall that the matrices in $\mathscr{U}_\beta$ have diagonal entries $\theta$ and $\phi$ which are divisors of $|\beta|_S$, and the upper right entry $\psi$ is an integer satisfying $0 \leq \psi \leq \phi$. After specifying $\theta$ and $\phi$, we leave $\psi$ as an unknown parameter and continue the computation, obtaining a form $G_\psi$ with coefficients in the ring $\Q[\psi]$. It is then simple to solve for those integral values, if any, of $\psi$ in the range $0 \leq \psi \leq \phi$ which yield binary forms with coefficients in $\OS$.
\end{remark}

\section{Quartic Forms up to $\OS^{0}$-Equivalence}\label{sec:QuarticS0Equiv}

Now that the $\OS$-equivalence classes of quartic forms have been enumerated, the next step is to create a list of quintic forms up to $\OS$-equivalence. From this, we create a complete list of quintic-linear pairs up to $\OS$-equivalence, which yields a complete list of $\OS^{0}$-equivalence classes of quartic forms.
\subsection{Quintic Forms up to $\OS$-Equivalence}\label{quintic forms s equiv}
As before, we let $\mathscr{M}_0$ denote a field system $(M_0, \dots, M_{m-1})$ with $\sum_i [M_i:\Q] = 4$, and let $M$ be the Galois closure of the compositum of the $M_i$. Let $\mathscr{M} = (\Q, M_0, \dots, M_{m-1})$. Here, we construct a set $\mathfrak{F}_5(\mathscr{M})$ which contains at least one representative of each $\OS$-equivalence class of quintic forms with the field system $\mathscr{M}$. This is done inductively from the equivalence classes in $\mathfrak{F}_4(\mathscr{M}_0)$, as in \cite[\S7]{Smart:1997}.  We briefly sketch the procedure here. Let $G(X,Z)$ be a degree 5 binary form with field system $\mathscr{M}$ and $D(G) \in \OS^{\times}$.  Then by Lemma \ref{lemma:SmartLemma2}, $G \sim_{\OS} LH$, where
\[ L(X, Z) = u_0 X + u_1 Z \in \OS[X, Z], \]
and the $\OS$-ideal $(u_0, u_1) = \OS$. Further, $H$ is $\OS$-equivalent to a form in $\mathfrak{F}_4(\mathscr{M}_0)$.   Since $L\sim_{\OS}Z$, without loss of generality we may assume $G\sim_{\OS}Z_UH$, where $U \in GL_2(\OS)$ and $H\in \mathfrak{F}_4(\mathscr{M}_0)$.  For each quartic $H\in\mathfrak{F}_4(\mathscr{M}_0)$, we seek all $Z_UH$ up to $\OS$-equivalence.

Let $\mathfrak{T}=\mathfrak{T}(M)$ be the set of $T$-unit solutions as in \S\ref{sec:QuarticSEquiv}. Since the compositum
of the fields in the field system $\mathscr{M}$ is $M$, the cross ratios $[i,j,k,\ell]$ for $Z_UH$ must lie in $\mathfrak{T}$.  Let
\[ Z_{U} = a_{0,0} X + a_{1,0} Z \text{ and } H = \prod_{i=1}^4(a_{0,i}X + a_{1,i}Z). \]
Since $H \in \mathfrak{F}_4(\mathscr{M}_0)$, the values of $a_{i,j}$ and $\Delta_{j,k}$ are known for $j \neq 0$. For any $Z_U H$, any choice of indices $\{i, j, k \} \subset \{1, 2, 3, 4 \}$ must satisfy
\begin{equation}\label{eq:cross ratio 0}
[0,i,j,k] = \frac{\Delta_{0,i}\Delta_{j,k}}{\Delta_{0,j}\Delta_{i,k}}.
\end{equation}
Note that $\Delta_{0,h} = a_{0,0}a_{1,h} - a_{1,0}a_{0,h}$.  Let $[0,i,j,k] = \tau \in \mathfrak{T}$. By rewriting \eqref{eq:cross ratio 0} we get the identity
\[ (\Delta_{i,k} a_{1,j} \tau - \Delta_{j,k} a_{1,i}) a_{0,0} - (\Delta_{i,k} a_{0,j} - \Delta_{j,k} a_{0,i}) a_{1,0} = 0 \]
in the unknowns $a_{0,0}$ and $a_{1,0}$.  The equation has solution set
\[ a_{0,0} = c(\Delta_{i,k} a_{0,j} - \Delta_{j,k} a_{0,i}), \hspace{.5in} a_{1,0} = c ( \Delta_{i,k} a_{1,j} \tau - \Delta_{j,k} a_{1,i}), \]
where $c$ is a parameter in $M$. Now choose $c$ such that $a_{0,0},a_{1,0} \in \OS$ with $(a_{0,0},a_{1,0}) = \OS$. Let $Z_{[\tau]} = a_{0,0}X + a_{1,0}Z$. Define
\[\mathfrak{F}_5(\mathscr{M}) := \{Z_{[\tau]} H : H \in \mathfrak{F}_4(\mathscr{M}_0), \tau \in \mathfrak{T}, Z_{[\tau]}\nmid H\}. \]
This set can be constructed exhaustively by looping over all possible choices of $H \in \mathfrak{F}_4(\mathscr{M}_0)$ and cross ratios from $\mathfrak{T}$, and attempting to compute $Z_{[\tau]}$, if possible. The resulting set contains at least one representative of every class in the set $\mathfrak{B}_5(\mathscr{M})$.

\subsection{Quartic Forms up to $\OS^{0}$-Equivalence}\label{subsec:QuarticS0Equiv}
For every $G\in\mathfrak{F}_5(\mathscr{M})$ and every $L\in \mathbb{Q}[X,Z]$ with $L \mid G$, consider the quintic-linear pair $(G,L)$.  By Lemma \ref{QL a and b}, we may find an $\mathcal{O}_S$-equivalent pair $(G_U,Z)$. Select one such pair for each possible choice of $G$ and $L$ to create a finite set $\mathfrak{Z}_{5,1}$ of quintic-linear pairs. After screening for equivalence, we may assume $\mathfrak{Z}_{5,1}$ contains a unique representative for each equivalence class in $\mathfrak{QL}$.

Now, as described in \S\ref{sec:PicardCurves}, a list of all Picard curves with good reduction away from $3$ up to $\mathbb{Q}$-isomorphism will consist of the curves
\[\mathcal{C}_{\alpha, F}: Y^3Z = \alpha F(X,Z),\]
where $(FZ,Z)\in\mathfrak{Z}_{5,1}$ and $\alpha\in\{1,3,9\}$.  A complete list of these curves appears in Tables  \ref{tab:QQK1}, \ref{tab:QK3}, and \ref{tab:QK2}.

\section{Solving $S$-unit equations}\label{sec:SolvingSUnit}

As a necessary step in our investigation, we must solve an $S$-unit equation over a finite collection of fields. In this section, let $K$ be a number field, and let $S$ be the set of places of $K$ which lie above the places in $S_\Q = \{3, \infty \}$. Let $\bmu_K$ denote the set of roots of unity in $K$, and let $w = \#\bmu_K$. Let $s$ be the number of finite places in $S$, $r_1$ the number of real embeddings of $K$, and $r_2$ the number of conjugate pairs of complex embeddings of $K$.  Let $r$ be the rank of $\OK^\times$, so $r = r_1 + r_2 - 1$. Let $t$ be the rank of $\OS^\times$, so $t = s + r$.    We fix a generator $\rho_0 \in \bmu_K$ for the roots of unity. The goal of this section is to find all unordered pairs $(\tau_0, \tau_1)$ which solve the following equation (the so-called ``$S$-unit equation''):
\begin{equation}
\label{eq:Sunit}
\tau_0 + \tau_1 = 1, \qquad \tau_i \in \OS^\times.
\end{equation}
In the previous sections, these solutions are referred to as $T$-units, however, the literature on this topic commonly refers to $S$-units so we adjust our notation accordingly for this section.

The general strategy for solving such a problem is to first fix a basis $\{\rho_i\}_{i=1}^t$ for the torsion-free part of the $\Z$-module $\OS^\times$. Then the set $\{\rho_i\}_{i=0}^t$ generates all of $\OS^\times$, and each solution may be described by the exponents $a_{j,i} \in \Z$ such that
\[ \tau_j = \prod_{i=0}^t \rho_i^{a_{j,i}}. \]
Thus, the tuple $(a_{j,0}, a_{j,1}, \dots, a_{j,t}) \in \Z/w\Z \times \Z^t$ uniquely encodes the value $\tau_j$. The number of solutions to \eqref{eq:Sunit} is known to be finite. To construct them explicitly, we proceed in three steps:
\begin{enumerate}
\item Compute an explicit bound $C_0$, via Baker's theory, for which any solution must satisfy
\[|a_{j,i}| \leq C_0. \]
Such a bound is typically much too large to be used in an exhaustive search.
\item Apply a LLL-type lattice argument to improve the bound by several orders of magnitude. Although this step is crucial, it alone is not enough to trivialize the problem; when the bound is $C_1$ and the rank of $\OS^\times$ is $r$, the search space has size roughly $(2C_1)^{2r}$, and this is still too large to carry out a brute-force search in practice.
\item Use the arithmetic of potential solutions (such as symmetry arising from Galois action) and an assortment of sieving methods to search for possible solutions efficiently.
\end{enumerate}
In this section, we discuss the method in detail and explain an improvement in step (2). (This improvement is a consequence of the special circumstances of the current problem and is unlikely to be available in general for solving $S$-unit equations.)

We now assume $K$ is a number field from Table \ref{table:nf_disc_3}. Let $\mathfrak{p}_0$ denote the unique prime of $K$ above $3$. Thus, $s = 1$ and $t = r_1 + r_2$. We have the real infinite places $\mathfrak{p}_1, \dots, \mathfrak{p}_{r_1}$, and the complex infinite places $\mathfrak{p}_{r_1 + 1}, \dots, \mathfrak{p}_t$. We let $\psi_1, \dots, \psi_{r_1}$ denote the $r_1$ distinct real embeddings $K \hookrightarrow \R$, and let
\[ \psi_{r_1 + 1}, \overline{\psi}_{r_1 + 1}, \dots, \psi_t, \overline{\psi}_t \]
denote the $2r_2$ distinct complex embeddings $K \hookrightarrow \C$, with the stipulation that no two distinct embeddings $\psi_{r_1 + k}$ and $\psi_{r_1 + \ell}$ may be conjugate. Once and for all, we relabel the embeddings so that each $\psi_i$ corresponds to the infinite place $\mathfrak{p}_i$. We consider the valuations associated to each 
place:
\[ | \alpha |_{\mathfrak{p}_i} := \begin{cases} 3^{- \ord_3 \alpha} & i = 0 \\
| \psi_i(\alpha) | & 1 \leq i \leq r_1 \\
| \psi_i(\alpha) |^2 & r_1 +1 \leq i \leq t.
\end{cases} \]
Henceforth, we write $\alpha^{(i)}$ as a shorthand for the image of $\alpha$ under the embedding $\psi_i \colon K \hookrightarrow \C$.

\subsection{Linear Forms on Logarithms}
The effective solution of $S$-unit equations rests on the observation from Baker's theory that the nonzero image of a linear form on linearly independent logarithms is bounded away from zero. We give an explicit statement here, to be used in a certain proof later on in this section. Further details may be found in \cite{Baker-Wustholz:1993}.

We let $h_0$ denote the standard logarithmic Weil height on $\P^t(K)$, defined as follows. For any $\mathbf{x} = (x_0 : \cdots : x_t) \in \P^t(K)$,
\[ h_0(\mathbf{x}) = \sum_\mathfrak{p} \log\, (\max_j \,\{ |x_j|_\mathfrak{p}^{n_\mathfrak{p}} \} ), \]
where the sum runs over all places of $K$, and $n_\mathfrak{p}$ denotes the local degree of $\mathfrak{p}$. By the product formula, $h_0(\mathbf{x})$ is independent of the particular choice of coordinates for $\mathbf{x}$. For any $\alpha \in K$, we take $h_0(\alpha) = h_0( (1 : \alpha) )$.

Let $B$ be a number field of degree $n_B$. Fix an embedding $\iota_B \colon B \hookrightarrow \C$. We will abuse notation and write $\log \alpha$ for $\log \iota_B(\alpha)$ for any $\alpha \in B$. We define the \emph{modified height} of $\alpha \in B$ by
\[ h'(\alpha) = \frac{1}{n_B} \max \left( h_0(\alpha), |\log \alpha|, 1 \right). \]
For a linear form in $t+1$ variables,
\[ L(z_0, \dots, z_t) = a_0 z_0 + \cdots + a_t z_t, \qquad a_i \in \Z, \]
we likewise define a modified height of $L$ by
\[ h'(L) := \max \bigl\{ 1, h_0( (a_0:a_1:\cdots:a_t) ) \bigr\}. \]
Since the $a_i \in \Z$, we have $|a_i|_\mathfrak{p} \leq 1$ for any finite place $\mathfrak{p}$. We may always take the coefficients to have no common prime divisor, and so at least one coefficient $a_i$ has $|a_i|_{\mathfrak{p}} = 1$ for each finite place $\mathfrak{p}$. Thus, the finite places make no contribution to $h'(L)$, and we have the following estimate: if $|a_j| \leq  H_1$ for all $j$, then $h'(L) \leq \log H_1$.

The following result gives an explicit bound, which we will use in the search for solutions to $S$-unit equations.
\begin{theorem}[Baker-W\"ustholz, {\cite[pg.~20]{Baker-Wustholz:1993}}]\label{thm:BW}
Let $L$ be a linear form in $t+1$ variables, and let $\rho_0, \dots, \rho_t \in \overline{\Q} - \{0, 1\}$. Let $B$ be the subfield of $\overline{\Q}$ generated by the $\rho_i$. If
\[ \Lambda = L(\log \rho_0, \log \rho_1, \dots, \log \rho_t) \neq 0, \]
then
\[ \log |\Lambda| > -C(t, n_B) h'(L) \prod_{j=0}^t h'(\rho_j), \]
where the constant $C(t, n_B)$ is defined by
\[ C(t, n_B) = 18 (t+2)! (t+1)^{(t+2)} (32 n_B)^{(t+3)} \log \left( 2(t+1) n_B \right). \]
\end{theorem}
Note that we may be sure $\Lambda \neq 0$ if the set $\{ \log \rho_i \}$ is linearly independent over $\Q$.

\subsection{Avoiding the finite place}\label{section:infinite place}

We introduce some definitions. For any $\tau \in \OS^\times$, we may write
\[ \tau = \prod_{j=0}^t \rho_j^{a_j}. \]
The \textit{minimal absolute value} of $\tau$ is
\[ \nu(\tau) := \min_j \{ |\tau|_{\mathfrak{p}_j} \}. \]
The \textit{extremal index} of $\tau$ is the largest index $j$ which achieves the minimal absolute value:
\[ \varepsilon(\tau) := \max \{ j : |\tau|_{\mathfrak{p}_j} = \nu(\tau) \}. \]

\begin{lemma}
Suppose $\tau \in \OS^\times$. Then $\varepsilon(\tau) = \varepsilon(\tau^{-1})$ if and only if $\tau \in \bmu_K$.
\end{lemma}
\begin{proof}
If $\tau$ is a root of unity, then every absolute value evaluates to $1$ and the extremal indices of $\tau$ and $\tau^{-1}$ necessarily coincide.

Conversely, if $\varepsilon(\tau) = \varepsilon(\tau^{-1})$, then the minimal absolute value for $\tau$ must be achieved by every $\mathfrak{p}_j$; under any other circumstance, the extremal indices of $\tau$ and $\tau^{-1}$ cannot coincide. Thus, $|\tau|_{\mathfrak{p}_i} = |\tau|_{\mathfrak{p}_j}$ for all primes in $S$. However, as $\tau \in \OS^\times$, $|\tau|_\mathfrak{q} = 1$ for every $\mathfrak{q} \not\in S$. By the product formula, it follows that $|\tau|_{\mathfrak{p}_i} = 1$ for every $\mathfrak{p}_i \in S$ also. This forces $\tau$ to be an algebraic integer, all of whose conjugates have absolute value $1$. So $\tau$ is a root of unity as claimed.
\end{proof}

The \emph{maximum exponent} of $\tau$ is the largest exponent in absolute value that appears when expressing $\tau$ in terms of the torsion-free basis:
\[ H(\tau) := \max_{j \geq 1} \{ |a_j| \}. \]
If $\mathbf{s} = (\tau_0, \tau_1)$ is a solution to the $S$-unit equation, we extend the above definitions as follows. The \emph{maximum exponent} of $\mathbf{s}$ is
\[ H(\mathbf{s}) = \max \{ H(\tau_0), H(\tau_1) \}. \]
We choose $b \in \{0, 1 \}$ in the following way. If $H(\tau_0) \neq H(\tau_1)$, then $b$ is chosen so that $H(\mathbf{s}) = H(\tau_b)$. If $H(\tau_0) = H(\tau_1)$, then $b$ is chosen so that $\varepsilon(\tau_b) \geq \varepsilon(\tau_{1-b})$. The \emph{extremal index} of $\mathbf{s}$ is defined to be $\varepsilon(\mathbf{s}) := \varepsilon(\tau_b)$. Essentially, the extremal index of $\mathbf{s}$ is the same quantity as the index $h$ from \cite[pg.~822]{Smart:1995}. The selection of $b$ is done so as to avoid $\varepsilon(\mathbf{s}) = 0$ if at all possible.

We hope to reduce the bound on $H(\mathbf{s})$ via a lattice reduction argument. However, this approach generally requires two different reduction arguments; a $p$-adic lattice argument for the case that the extremal index of $\mathbf{s}$ corresponds to a finite place of $S$, and a complex lattice argument for the case that the index corresponds to an infinite place. The following technical lemma will allow us to reduce the search to those solutions whose extremal index occurs at an infinite place. Thus, we will only need the complex lattice argument in the sequel.

Let $\mathbf{s} = (\tau_0, \tau_1)$ be a solution to the $S$-unit equation. The \emph{cycle} associated to $\mathbf{s}$ is the set $C(\mathbf{s}) = \{\mathbf{s}, \mathbf{s}', \mathbf{s}'' \}$, where
\[ \mathbf{s}' = (\tau_0^{-1}, -\tau_0^{-1} \tau_1), \qquad \mathbf{s}'' = (-\tau_0 \tau_1^{-1}, \tau_1^{-1}). \]
It is easy to see that $\mathbf{s}'$ and $\mathbf{s}''$ are also solutions to the $S$-unit equation. Moreover, since a solution is an unordered pair, the cycle is unaffected by swapping $\tau_0$ and $\tau_1$.

\begin{lemma}\label{lemma:infinite place}
Let $\mathbf{s}$ be a solution to the $S$-unit equation. Then at least one solution in $C(\mathbf{s})$ has an extremal index which is not equal to zero (and hence corresponds to an infinite place).
\end{lemma}
\begin{proof}
Suppose $\mathbf{s} = (\tau_0, \tau_1)$. If both $\tau_0, \tau_1 \in \bmu_K$, then $\varepsilon(\mathbf{s}) = t > 0$. So we may now assume at least one $\tau_i$ is not in $\bmu_K$. Choose $b \in \{0, 1 \}$ as above, and set $d = 1 - b$. We may be sure that $\tau_b$ is not a root of unity, since $H(\tau) = 0$ if and only if $\tau \in \bmu_K$. For the sake of a contradiction, suppose $\varepsilon(\mathbf{s}) = \varepsilon(\mathbf{s}') = \varepsilon(\mathbf{s}'') = 0$, and write $\mathbf{s} = (\tau_b, \tau_d)$.

Since $H(\mathbf{s}) = H(\tau_b)$, we know $\varepsilon(\tau_b) = 0$. But $\tau_b \not\in \bmu_K$, so $\varepsilon(\tau_b^{-1}) \neq 0$. The other solutions in the cycle $C(\mathbf{s})$ are $\mathbf{t} = (\tau_b^{-1}, -\tau_b^{-1} \tau_d)$ and $\mathbf{t}' = (\tau_d^{-1}, -\tau_b \tau_d^{-1})$. If $H(\tau_b^{-1}) \geq H( - \tau_d \tau_b^{-1})$, then $\varepsilon(\mathbf{t}) > 0$. So we may assume $H(\tau_b^{-1}) < H(-\tau_d \tau_b^{-1})$. But from this we have
\[ H(\mathbf{t}) = H(-\tau_d \tau_b^{-1}) > H(\tau_b^{-1}) = H(\tau_b) = H(\mathbf{s}). \]
Moreover, the values in $\mathbf{t}'$ satisfy
\[ H(\tau_d^{-1}) = H(\tau_d) \leq H(\tau_b) = H(\mathbf{s}), \]
and
\[ H( - \tau_b \tau_d^{-1} ) = H( - \tau_d \tau_b^{-1}) = H(\mathbf{t}) > H(\mathbf{s}). \]
Thus, $H(\mathbf{t}') = H(\mathbf{t}) > H(\mathbf{s})$. However, $C(\mathbf{t}) = C(\mathbf{s})$, so by repeating the same argument we may conclude that $H(\mathbf{s}) = H(\mathbf{t}') > H(\mathbf{t})$, a contradiction.
\end{proof}

\begin{remark}
Notice that we are not using the fact that $\mathfrak{p}_0$ is finite. Fix any place $\mathfrak{q} \in S$. The argument can easily be modified (simply by re-indexing the places in $S$) to prove that any cycle has at least one solution whose extremal index does not correspond to $\mathfrak{q}$.
\end{remark}

\begin{remark}
In a later section, we will find a strong bound for $H(\mathbf{s})$ in the special case that $\mathbf{s}$ has an extremal index which is non-zero. Of course, recovering $C(\mathbf{s})$ from $\mathbf{s}$ is trivial, and so this is enough to capture all $S$-unit equation solutions.
\end{remark}

\subsection{A Computable Bound}

The goal of this section is to prove the following proposition:

\begin{proposition}
Suppose $\mathbf{s}$ is a solution to the $S$-unit equation. If $\varepsilon(\mathbf{s}) \neq 0$, then $H(\mathbf{s}) < C_0$, an effectively computable constant which depends only on $K$ and $S$.
\end{proposition}
The proof is similar to that of \cite[Lemma 4]{Smart:1995}, although the determination of the constant $C_0$ (labeled $K_1$ in Smart's presentation) will differ slightly.
\begin{proof}
Possibly after relabeling, we may assume $\mathbf{s} = (\tau_0, \tau_1)$, $b = 0$, and
\[ \varepsilon(\mathbf{s}) = \varepsilon(\tau_0) = h > 0, \]
for some $1 \leq h \leq r_1 + r_2$. For $j \in \{0, 1 \}$, write
\[ \tau_j' = \prod_{i=1}^t \rho_i^{a_{j,i}}, \]
so that $\tau_j = \rho_0^{a_{j,0}} \tau_j'$. Consider the set $S_\infty = \{\mathfrak{p}_1, \dots, \mathfrak{p}_{r_1 + r_2} \}$ of infinite places of $S$, the first $r_1$ being the real places. As in \cite{Smart:1995}, we define the following constants, depending on the value $1 \leq h \leq r_1 + r_2$:
\[
\begin{split}
c_{11}(h) & := \begin{cases} \displaystyle \frac{ \log(4) }{c_3} & \mbox{if
    $\mathfrak{p}_h$ is real} \\ \displaystyle \frac{ 2\log(4) }{c_3} &
  \mbox{if $\mathfrak{p}_h$ is complex,} \end{cases} \\
c_{12}(h) & := 2, \\
c_{13}(h) & := \begin{cases} c_3 & \mbox{if $\mathfrak{p}_h$ is real} \\
  c_3/2 & \mbox{if $\mathfrak{p}_h$ is complex.} \end{cases}
\end{split}
\]
The constant $c_3$ is defined in \cite[pg.~823]{Smart:1995}, and is derived from the regulator for the $S$-units of $K$. For this project we always have $0.1 < c_3 < 0.3$. From \cite[Lemma 2]{Smart:1995}, we know that $|\tau_0|_{\mathfrak{p}_h} \leq e^{-c_3 H(\mathbf{s})}$. If we suppose $H(\mathbf{s}) \geq c_{11}(h)$, then consequently
\[ |\tau_1 - 1|_{\mathfrak{p}_h} =|\tau_0|_{\mathfrak{p}_h} \leq \frac{1}{4}. \]
For any $z \in \C$ satisfying $|1-z| \leq \frac{1}{4}$, we have the estimate
\[ |\log z| \leq 2|1-z|, \]
and so
\begin{equation}
\label{eq:log_tau_1_UB}
|\log \tau_1^{(h)}| \leq 2 | \tau_1 - 1|_{\mathfrak{p}_h} = 2|\tau_0|_{\mathfrak{p}_h} \leq 2e^{-c_3 H(\mathbf{s})}.
\end{equation}
On the other hand, let $\log$ denote the principal branch of the logarithm. We have
\begin{equation}
\label{eq:log_tau_1}
\left| \log \tau_1^{(h)} \right| = \left| \log \left( (\rho_0^{a_{0,1}})^{(h)} \right) + \sum_{j=1}^t a_{j,1} \log \rho_j^{(h)} + A_1 2\pi \sqrt{-1} \right|,
\end{equation}
where the term $A_1 2\pi \sqrt{-1}$ arises as a consequence of demanding principal values for the logarithms. In order to take advantage of Baker's theory of linear forms of logarithms, we need the various expressions to be independent over $\Q$, but this is not true for $\log \rho_0^{(h)}$ and $2\pi\sqrt{-1}$. We must combine these terms, while still maintaining a bound on the coefficients. We explain this step now.

Suppose $A = \prod_{i=1}^N \alpha_i$, and we take principal arguments in the range $[-\pi, \pi)$. Then $\sum \arg \alpha_i \in [-N\pi, N\pi)$, and
\[ \arg A = \sum \arg \alpha_i + m \cdot 2 \pi, \]
where $|m| \leq N/2$. We have
\[ \log A = \sum \log \alpha_i + m \cdot 2\pi \sqrt{-1}, \]
again with $|m| \leq N/2$. Thus, in \eqref{eq:log_tau_1}, we have the bound
\[ |A_1| \leq \frac{1}{2} (1 + \sum_{j=1}^t |a_{j,1}|) \leq \frac{1}{2} (1 + t H(\mathbf{s})). \]
Inside $\C$, let $\zeta := \exp(2\pi\sqrt{-1}/w)$, so that $\zeta = \rho_0^{(h)}$. There is an integer $\ell$ with $-w/2 < \ell \leq w/2$ for which
\[ (\rho_0^{a_{0,1}})^{(h)} = \zeta^\ell. \]
We adjust our branch of $\log$ slightly. We choose a branch whose argument always lies in the range $(-\pi + \varepsilon, \pi + \varepsilon)$ for some sufficiently small $\varepsilon > 0$, such that each of the expressions $\log \rho_j$, for $j > 0$, and $\log \zeta^\ell$ agrees with its principal value. Now $-1 = \zeta^{w/2}$ and so
\[
\log \left( (\rho_0^{a_{0,1}})^{(h)} \right) + A_1 2\pi\sqrt{-1} = \log \zeta^\ell + 2A_1 \log \zeta^{w/2} = (\ell + w A_1) \log \zeta. \]
Set $a_{0,1}' = \ell + w A_1$. We have the bound $|a_{0,1}'| \leq \frac{w}{2}(t+2)H(\mathbf{s})$, and we may rewrite \eqref{eq:log_tau_1} as
\[ \left| \log \tau_1^{(h)} \right| = \left| a_{0,1}' \log \zeta + \sum_{j=1}^t a_{j,1} \log \rho_j^{(h)} \right|, \]
where the coefficients $a_{0,1}'$ and $a_{j,1}$ have been given explicit bounds.

Let $L$ denote the linear form in $t+1$ variables
\[ L(z_0, z_1, \dots, z_t) := a_{0,1}' z_0 + \sum_{j=1}^t a_{j,1} z_i. \]
Then we have
\[ \log \tau_1^{(h)} = L \left( \log \zeta, \log \rho_1^{(h)}, \dots, \log \rho_t^{(h)} \right), \]
and the coefficients of $L$ are all bounded by $\frac{w}{2}(t+2)H(\mathbf{s})$. From this observation, we see
\[ h'(L) \leq \log \left( \frac{w}{2}(t+2) H(\mathbf{s}) \right). \]
Now let us define
\[ c_{14}'(h) := C(t + 1, n_K) \cdot \frac{1}{n_K} \cdot h'(\zeta) \prod_{j=1}^t h'(\rho_j^{(h)}). \]
By Theorem \ref{thm:BW}, we have
\[ | \log \tau_1^{(h)} | > \exp \left( -c_{14}'(h) \log \left( \frac{w}{2}(t+2)H(\mathbf{s}) \right) \right). \]
Combining with the inequality \eqref{eq:log_tau_1_UB}, we obtain:
\[ 2e^{-c_3 H(\mathbf{s})} > \exp \left( -c_{14}'(h) \log \left( \frac{w}{2}(t+2)H(\mathbf{s}) \right) \right). \]
Rewriting this inequality as one comparing $H(\mathbf{s})$ and $\log H(\mathbf{s})$, and then applying a lemma of Peth\"o and de~Weger \cite[Lemma 2.2]{Petho-deWeger:1986}, we obtain the bound
\[ H(\mathbf{s}) \leq c_{15}'(h) := \frac{2}{c_{13}(h)} \left( \log c_{12}(h) + c_{14}'(h) \log \left( \frac{w(t+2)c_{14}'(h)}{2c_{13}(h)} \right) \right). \]
This almost gives us the desired bound; however, we still do not know \emph{which} index is given by $h$. Finally, we take
\[ C_{11} := \max_{1 \leq j \leq t} \, \{ c_{11}(j) \}, \qquad C_{15} := \max_{1 \leq j \leq t} \, \{ c_{15}'(j) \}, \qquad C_0 := \max\,\{C_{11}, C_{15}\}, \]
and we may be sure $H(\mathbf{s}) \leq C_0$, as desired.
\end{proof}

The computation of these bounds is routine. The computation of the constant $c_3$ is explained in \cite[pg.~823]{Smart:1995}. Table \ref{table:C0} gives a (mild over-estimate for a) value of $C_0$ for each field under consideration.

\begin{table}[!ht]
{\renewcommand{\arraystretch}{1.2}
\caption{Bounds $C_0$, $C_0'$ for $H(\mathbf{s})$ by field.}
\label{table:C0}
\begin{tabular}{crr}
\toprule
Field & $C_0 \qquad \quad {}$ & $C_0'$ \\
\midrule
$K_0$ & $4.916825 \times 10^{9\phantom{0}} $ & $3$ \\
$K_1$ & $8.018712 \times 10^{9\phantom{0}} $ & $5$ \\
$K_2$ & $2.067269 \times 10^{19}$ & $217$ \\
$K_3$ & $1.957261 \times 10^{15}$ & $49$ \\
$L_3$ & $2.137374 \times 10^{19}$ & $243$ \\
\bottomrule
\end{tabular}
}
\end{table}

\subsection{Improvement of $C_0$ by lattice reduction}

As in \cite[\S4, pg.~828-830]{Smart:1995}, we use the LLL algorithm to provide a substantial reduction in the bound $C_0$. For each embedding $\psi_i \colon K \hookrightarrow \C$, we explain how to reduce the bound $C_0$ under the assumption that the extremal index of some solution $\mathbf{s}$ is $\varepsilon(\mathbf{s}) = h \neq 0$. The maximum of these reduced bounds yields a replacement $C_0'$ for the constant $C_0$.

\begin{remark}
The following calculation is sensitive to the choice of ordered $\Z$-basis $\rho_1, \dots, \rho_t$ for the free part of $\OS^\times$. We require that the $\rho_i$ are always chosen so that in every embedding $\psi_h$, at least one value $\rho_j^{(h)}$ for $1 \leq j \leq t$ is not a positive real number. This is not a serious restriction, but it does allow us to always use the LLL algorithm as described by Smart in \cite[\S4]{Smart:1995}.
\end{remark}
Fix a particular embedding $\psi_h \colon K \hookrightarrow \C$ for some $h \neq 0$, and consider
\begin{equation}
\Lambda := \sum_{j=0}^t a_j \kappa_j,
\end{equation}
where $a_0 = a_{0,1}'$, $a_j = a_{0,j}$ for $j > 0$, $\kappa_j = \log \rho_j^{(h)}$ for $j > 0$. For $\kappa_0$ we have
\[ \kappa_0 = \log \zeta = \frac{2 \pi \sqrt{-1}}{w}. \]
We relabel the $\kappa_j$ as necessary to ensure that $\Re\,\kappa_t \neq 0$.

We may assume $|a_j| \leq C_0$ for $j > 0$; we have $|a_0| \leq \frac{w}{2}(t+2)C_0$. For any $\gamma \in \R$, let $[\gamma]$ denote the integer closest to $\gamma$. For any $z \in \C$, $\Re z$ and $\Im z$ denote the real and imaginary parts of $z$, respectively. Let $C \gg 0$ be a large constant, whose value we will make precise in a moment. We define
\[ \begin{split}
\Phi_0 & := \sum_{j=1}^t a_j [C\, \Re \kappa_j], \\
\Phi_1 & := \sum_{j=1}^t a_j [C\, \Im \kappa_j] + a_0[C \cdot \tfrac{2\pi}{w}].
\end{split} \]
Let $\mathscr{B}$ be the following integral matrix:
\[ \mathscr{B} := \left( \begin{array}{cccccc}
1 & 0 & \cdots & 0 & 0 & 0 \\
0 & 1 & \cdots & 0 & 0 & 0 \\
\vdots & \vdots & & \vdots & \vdots & \vdots \\
0 & 0 & \cdots & 1 & 0 & 0 \\

[C\, \Re \kappa_1] & [C\, \Re \kappa_2] & \cdots & [C\, \Re \kappa_{t-1}] & [C\, \Re \kappa_t] & 0 \\

[C\, \Im \kappa_1] & [C\, \Im \kappa_2] & \cdots & [C\, \Im \kappa_{t-1}] & [C\, \Im \kappa_t] & [C \cdot \tfrac{2\pi}{w}]
\end{array} \right). \]
Let $\mathscr{L}$ be the lattice generated by the columns of $\mathscr{B}$.

From \cite[Prop.~1.11]{LLL:1982} (or \cite[Lemma 3.4]{deWeger:1989}, where the proof is given in slightly more detail), the Euclidean length of any nonzero lattice element in $\mathscr{L}$ is bounded below by $C_1 := 2^{-t/2}\| \mathbf{b}_0 \|$, where $\mathbf{b}_0$ is the shortest vector in the LLL-reduced basis for $\mathscr{L}$. We define
\[ \begin{split}
S_\mathscr{L} & := \left( C_1^2 - (t-1)C_0^2 \right)^{1/2}, \\
T_\mathscr{L} & := \frac{1}{2}\left( w + 2 + \sqrt{2} \right) tC_0.
\end{split} \]
Notice that $C_1$ (and so $S_\mathscr{L}$) depends on the choice of $C$, but $C_0$ and $T_\mathscr{L}$ do not. As a practical matter, we choose $C \approx C_0^{(t+1)/2}$, compute the reduced basis for $\mathscr{L}$, and check whether
\begin{equation}\label{eq:C1-condition}
C_1^2 > T_\mathscr{L}^2 + (t-1)C_0^2,
\end{equation}
as this implies $S_\mathscr{L} - T_\mathscr{L} > 0$. If not, we replace $C$ by $2C$ and try again. Although there is no proof that this procedure will terminate, we did find a constant $C$ for each   field which forced $S_\mathscr{L} - T_\mathscr{L} > 0$.  This is beneficial in light of the following result:
\begin{lemma}
Suppose there exists $C > 0$ such that \eqref{eq:C1-condition} holds. Then every solution $\mathbf{s}$ to the $S$-unit equation with $\varepsilon(\mathbf{s}) = h \neq 0$ satisfies
\[ H(\mathbf{s}) \leq C_0' := \max \left\{ C_{11}, \left\lfloor \frac{1}{c_{13}(h)} \bigl( \log (C c_{12}(h)) - \log (S_\mathscr{L} - T_\mathscr{L}) \bigr) \right\rfloor \right\}. \]
\end{lemma}
\begin{proof}
The proof is similar to the proof of Lemma 6 in \cite{Smart:1995}.
\end{proof}

\subsection{Sieving for Solutions}

Since recovering every solution in a cycle is trivial if one solution is known, we are now left with the following problem: Given a bound $C_0'$, determine all solutions $\mathbf{s} = (\tau_0, \tau_1)$ to the $S$-unit equation, where
\[ \tau_i = \rho_0^{a_{0,i}} \prod_{j=1}^t \rho_j^{ a_{j,i} }, \qquad 0 \leq a_{0,i} < w, \quad |a_{j,i}| \leq C_0'. \]
In the worst case scenario, when $M = L_3$, the rank of $\mathscr{O}_{M,S}^\times$ is $3$, and so the number of possible solutions to check is roughly
\[ (2C_0' + 1)^3\cdot w \approx 4.62 \times 10^8. \]
This is still not manageable for a brute force search. So a sieving method is used to accelerate the exhaustive search for solutions. We describe briefly the approach, which is similar to some of the techniques described by Smart in \cite[\S8]{Smart:1997}. The entire computation was carried out in Sage \cite{Sage} in a matter of several minutes using a project on SageMathCloud.

Let $\tau = \brho^\mathbf{a}$ with
\[ \mathbf{a} = (a_0, \dots, a_t), \qquad 0 \leq a_0 < w, \quad |a_j| \leq C_0' \text{ for } j > 0. \]
Let $q \in \Z$ be a prime which is completely split in $\mathscr{O}_M$, say
\[ q\mathscr{O}_M = \mathfrak{q}_0 \cdots \mathfrak{q}_{n-1}, \quad n = [M:\Q]. \]
For each $0 \leq i < n$, we let $\rho_{(i)}$ denote the tuple
\[ \rho_{(i)} = ( \overline{\rho}_0, \overline{\rho}_1, \dots, \overline{\rho}_t ) \in \bigl( \mathscr{O}_M / \mathfrak{q}_i \bigr)^{t+1} \cong \F_q^{t+1}, \]
where the overline denotes reduction modulo $\mathfrak{q}_i$. If $\tau_i = \brho^{\mathbf{a}_i}$ satisfy $\tau_0 + \tau_1 = 1$, we obtain $r$ equations of the form
\[ \overline{\tau}_0 + \overline{\tau}_1 = 1 \in \F_{\mathfrak{q}_i}. \]
Now, suppose that the exponent vector $\mathbf{a}_0$ for $\tau_0$ has been \emph{fixed} modulo $q-1$. Each of the $r$ equations places a possibly \emph{different} set of conditions on $\mathbf{a}_1$ modulo $q-1$. We write a procedure to catalog the approximately $q^t$ results of the form: ``if $\mathbf{a}_0$ modulo $(q-1)$ is given by the vector $\mathbf{b}_0 \in \F_q^{t+1}$, then $\mathbf{a}_1$ modulo $(q-1)$ must come from  an explicit finite set.'' Moreover, the given finite set is often empty for many choices of $\mathbf{b}_0$. So in practice, we obtain congruence conditions modulo $(q-1)$ on the exponent vectors of $\tau_0, \tau_1$.

We now select a finite list of rational primes, $q_0, q_1, \dots, q_N$, each of which splits completely in $M$, and such that
\[ \lcm \bigl( q_0 - 1, q_1 - 1, \dots, q_N - 1 \bigr) > 2C_0'. \]
Then (essentially by a Chinese remainder theorem type argument), we restrict the possible values for $\mathbf{a}_0$ to a set of searchable size, and for each such $\mathbf{a}_0$, check whether the collected congruence conditions on $\mathbf{a}_1$ actually yield a solution to the $S$-unit equation. With this approach, we were able to produce all solutions to the $S$-unit equation for any field $M$ which is the Galois closure of the compositum of the fields in a field system attached to a quartic binary form with good reduction outside $S_\Q = \{3, \infty\}$.

For reference, we record the number of $S$-unit equation solutions we found for each of the fields $K_i, L_3$ in Table \ref{table:num_SU_sols}.
\begin{table}[!t]
{\renewcommand{\arraystretch}{1.2}
\caption{Number of $S$-unit equation solutions $\mathbf{s}$ by field.}
\label{table:num_SU_sols}
\begin{tabular}{cc}
\hline
Field & Solutions \\
\hline
$K_0$ & \phantom{0}0 \\
$K_1$ & \phantom{0}4 \\
$K_2$ & 72 \\
$K_3$ & \phantom{0}0 \\
$L_3$ & 25 \\
\hline
\end{tabular}
}
\end{table}
The fact that there are no solutions $\mathbf{s}$ in the field $K_0 = \Q$ has an important consequence.
\begin{lemma}
Let $F(X,Z) \in \Q[X,Z]$ be a binary quartic form with good reduction away from $3$. The the field system of $F(X,Z)$ cannot be $(K_0, K_0, K_0, K_0)$ or $(K_1, K_1)$.
\end{lemma}
\begin{proof}
Without loss of generality, we may assume an $S$-proper factorization for $F$. Now, the principal observation is that the equation \eqref{delta equation} must hold. If the field system is $(K_0, K_0, K_0, K_0)$, then the cross ratios in the equation represent $S$-unit equation solutions lying in $K_0 = \Q$, and no such solutions exist. On the other hand, if the field system is $(K_1, K_1)$, then let $\sigma$ denote the nontrivial element of $\mathrm{Gal}(K_1/\Q)$. Then the cross ratio
\[ [0, 1, 2, 3] = \frac{ \Delta_{0,1} \Delta_{2,3} }{ \Delta_{0,2} \Delta_{1,3} } \]
satisfies
\[ [0, 1, 2, 3]^\sigma = [1, 0, 3, 2] = [0, 1, 2, 3], \]
so again this provides a contradiction: namely, a solution to the $S$-unit equation which lies in $\Q$.
\end{proof}

\section{Results and an Application}\label{sec:Conclusions}

\subsection{Results} Using the process described in this paper, we found $63$ distinct $\Q$-isomorphism classes of Picard curves defined over $\Q$ with good reduction outside of $3$. Tables \ref{tab:QQK1}, \ref{tab:QK3}, \ref{tab:QK2} give the coefficients of integral affine models of the form
\[ y^3 = x^4 + a_3 x^3 + a_2 x^2 + a_1 x + a_0, \qquad a_i \in \Z. \]
In each case, the given model itself has good reduction away from $3$. Each table lists the curves possessing one particular field system. Further, each table is broken into blocks of $3$ curves each. The curves in each block are cubic twists of one another, and are isomorphic over $\Q(\sqrt[3]{3})$.

\begin{table}[!ht]
\caption{Picard curves with field system $(K_0, K_0, K_1)$}
\label{tab:QQK1}
\begin{tabular}{rrrr}
\toprule
$a_3$ & $a_2$ & $a_1$ & $a_0$ \\
\midrule
$ \phantom{-000}0 $ & $ \phantom{-000}0 $ & $ \phantom{-000}1 $ & $ \phantom{-000}0 $  \\
$     0 $  & $     0 $  & $    27 $  & $     0 $  \\
$     0 $  & $     0 $  & $   729 $  & $     0 $  \\
\midrule
$     2 $  & $     2 $  & $     1 $  & $     0 $  \\
$     2 $  & $     6 $  & $     5 $  & $   -14 $  \\
$    14 $  & $   114 $  & $   455 $  & $  -584 $  \\
\midrule
$    -2 $  & $     0 $  & $     1 $  & $    -2 $  \\
$    -2 $  & $   -12 $  & $    13 $  & $  -140 $  \\
$    50 $  & $   816 $  & $  4775 $  & $ -5642 $  \\
\bottomrule
\end{tabular}
\end{table}

\begin{table}[!ht]
\caption{Picard curves with field system $(K_0, K_3)$}
\label{tab:QK3}
\begin{tabular}{rrrr}
\toprule
$a_3$ & $a_2$ & $a_1$ & $a_0$ \\
\midrule
$ \phantom{-000}0 $  & $ \phantom{-000}0 $  & $ \phantom{-000}3 $ & $ \phantom{-000}0 $ \\
$     0 $  & $     0 $  & $    81 $  & $     0 $ \\
$     0 $  & $     0 $  & $  2187 $  & $     0 $ \\
\midrule
$     0 $  & $     0 $  & $     9 $  & $     0 $ \\
$     0 $  & $     0 $  & $   243 $  & $     0 $ \\
$     0 $  & $     0 $  & $  6561 $  & $     0 $ \\
\midrule
$    -2 $  & $     0 $  & $    11 $  & $     8 $ \\
$   -18 $  & $   108 $  & $    27 $  & $     0 $ \\
$   -54 $  & $   972 $  & $   729 $  & $     0 $ \\
\midrule
$    12 $  & $    -6 $  & $     1 $  & $     0 $  \\
$    36 $  & $   -54 $  & $    27 $  & $     0 $  \\
$   112 $  & $  -156 $  & $    85 $  & $   352 $  \\
\bottomrule
\end{tabular}
\end{table}

\begin{table}[!ht]
\caption{Picard curves with field system $(K_0, K_2)$}
\label{tab:QK2}
\begin{tabular}{ccc}
   \begin{tabular}{rrrr}
\toprule
$a_3$ & $a_2$ & $a_1$ & $a_0$ \\
\midrule
$    -3 $  & $     \phantom{-000}0 $  & $     1 $  & $     0 $  \\
$    -5 $  & $   -21 $  & $     4 $  & $    19 $  \\
$   -31 $  & $    87 $  & $   644 $  & $  -701 $  \\
\midrule
$     \phantom{-000}0 $  & $    -3 $  & $     1 $  & $     0 $  \\
$     0 $  & $   -27 $  & $    27 $  & $     0 $  \\
$     0 $  & $  -243 $  & $   729 $  & $     0 $  \\
\midrule
$    -2 $  & $    -3 $  & $     1 $  & $     1 $  \\
$   -14 $  & $    33 $  & $    31 $  & $   -17 $  \\
$   -38 $  & $   177 $  & $  1309 $  & $  -284 $  \\
\midrule
$    -1 $  & $    -3 $  & $     2 $  & $     1 $  \\
$    13 $  & $    33 $  & $   -50 $  & $   -71 $  \\
$    31 $  & $    87 $  & $ -2102 $  & $ -2159 $  \\
\midrule
$    -2 $  & $    -3 $  & $     3 $  & $     3 $  \\
$    18 $  & $    81 $  & $    27 $  & $     0 $  \\
$   -50 $  & $   573 $  & $   571 $  & $   -53 $  \\
\midrule
$    -5 $  & $    -3 $  & $     4 $  & $     1 $  \\
$   -23 $  & $    87 $  & $     4 $  & $  -107 $  \\
$   -57 $  & $   216 $  & $  3051 $  & $ -3078 $  \\
\midrule
$    -6 $  & $     3 $  & $     1 $  & $     0 $  \\
$   -18 $  & $    27 $  & $    27 $  & $     0 $  \\
$   -58 $  & $   411 $  & $    77 $  & $  -431 $  \\
\bottomrule
   \end{tabular}
& \phantom{0} &
   \begin{tabular}{rrrr}
\toprule
$a_3$ & $a_2$ & $a_1$ & $a_0$ \\
\midrule
$     \phantom{-000}3 $  & $    -6 $  & $     \phantom{-000}1 $  & $     0 $ \\
$   -13 $  & $   -21 $  & $    50 $  & $   -17 $ \\
$   -31 $  & $  -399 $  & $   158 $  & $   271 $ \\
\midrule
$     9 $  & $     6 $  & $     1 $  & $     0 $ \\
$   -23 $  & $   -21 $  & $     4 $  & $     1 $ \\
$   -69 $  & $  -189 $  & $   108 $  & $    81 $ \\
\midrule
$     6 $  & $    -9 $  & $     3 $  & $     0 $ \\
$   -22 $  & $   -21 $  & $    23 $  & $    19 $ \\
$    78 $  & $   459 $  & $   135 $  & $  -162 $ \\
\midrule
$     0 $  & $    -9 $  & $     9 $  & $     0 $ \\
$     0 $  & $   -81 $  & $   243 $  & $     0 $ \\
$     0 $  & $  -729 $  & $  6561 $  & $     0 $ \\
\midrule
$     9 $  & $    18 $  & $    -9 $  & $     0 $ \\
$    27 $  & $   162 $  & $  -243 $  & $     0 $ \\
$    93 $  & $  2241 $  & $  4482 $  & $ -4293 $ \\
\midrule
$    24 $  & $     3 $  & $    -1 $  & $     0 $ \\
$    72 $  & $    27 $  & $   -27 $  & $     0 $ \\
$   216 $  & $   243 $  & $  -729 $  & $     0 $ \\
\midrule
$     3 $  & $   -24 $  & $     1 $  & $     0 $ \\
$     9 $  & $  -216 $  & $    27 $  & $     0 $ \\
$    27 $  & $ -1944 $  & $   729 $  & $     0 $ \\
\bottomrule   \end{tabular}
\\
\end{tabular}
\end{table}

\subsection{Relation to Ihara's Question}

Let $\ell$ be a rational prime number, and let $K$ be a number field. We let $K(\bmu_{\ell^\infty})$ denote the maximal $\ell$-cyclotomic extension of $K$, and let $\ten = \ten(K,\ell)$ denote the maximal pro-$\ell$ extension of $K(\bmu_{\ell^\infty})$ unramified outside $\ell$. An abelian variety $A/K$ is said to be \emph{heavenly} at $\ell$ if the field $K(A[\ell^\infty])$ is a subfield of $\ten$.\footnote{This choice of notation and terminology is explained in more detail in \cite{Rasmussen-Tamagawa:2013}. The Japanese kanji $\ten$ and $\san$ are pronounced \emph{ten} and \emph{san}, respectively.}

Let $G_\Q$ denote the absolute Galois group $\Gal(\overline{\Q}/\Q)$, and let $X = \P^1_{01\infty}$, the projective line (over $\Q$) with three points deleted. There is a canonical outer Galois representation attached to the pro-$\ell$ algebraic fundamental group of $\overline{X} = X \times_\Q \overline{\Q}$,
\[ \Phi_\ell \colon G_\Q \longrightarrow \Out \bigl( \pi_1^\text{pro-$\ell$}(\overline{X}) \bigr). \]
We let $\san = \san(\Q, \ell)$ denote the fixed field of the kernel of $\Phi_\ell$. It is well known that $\san \subseteq \ten$, and Ihara has asked, specifically in the case $K = \Q$, whether or not the two fields coincide.

Now, suppose $\mathcal{C}/\Q$ is a projective curve which satisfies the following two conditions:
\begin{enumerate}
\item[(I1)] $\mathcal{C}/\Q$ has good reduction away from $\ell$,
\item[(I2)] An open subset of $\mathcal{C}$ appears in the pro-$\ell$ \'etale tower over $\P^1_{01\infty}$.
\end{enumerate}
To rephrase the second condition, there exists, over some field extension, a morphism $f \colon \mathcal{C} \to \P^1$ branched only over $\{0, 1, \infty \}$, and whose Galois closure is of $\ell$-power degree. Let $J$ denote the Jacobian variety of $\mathcal{C}$. Then Anderson and Ihara have given a necessary condition on the morphism $f$ \cite[Cor.~3.8.1]{Anderson-Ihara:1988} which implies $\Q(J[\ell^\infty]) \subseteq \san$.

The Picard curves enumerated in the present article all have good reduction away from $3$, and naturally admit a degree $3$ morphism $f_0 \colon \mathcal{C} \to \P^1$. Thus, in the spirit of \cite{Papanikolas-Rasmussen:2007} or \cite{Rasmussen:2012}, we may use the criterion of Anderson and Ihara if there exists a morphism $g \colon \P^1 \to \P^1$ such that $f = g \circ f_0$ satisfies the conditions (I1) and (I2) above. For example, on each of the curves
\begin{equation}\label{eq:simple_Picard}
y^3 = x^4 + 3^s x, \qquad 0 \leq s \leq 8,
\end{equation}
there exists a degree $9$ morphism $f \colon \mathcal{C} \to \P^1$ defined in affine coordinates by
\[ (x, y) \mapsto -3^{-s} x^3. \]
It is necessarily Galois over $K = \Q(\bmu_3)$, as it corresponds to a composition of degree $3$ Kummer extensions: $K(t) \hookrightarrow K(t^{1/3}) \hookrightarrow K( (t^{4/3} + 3^s t^{1/3})^{1/3})$. The verification of Anderson and Ihara's criteria is immediate, and so we obtain the following.
\begin{corollary}
Let $\mathcal{C}$ be one of the nine Picard curves over $\Q$ with good reduction away from $3$, possessing an affine integral model of the form \eqref{eq:simple_Picard}. Let $J$ denote the Jacobian of $\mathcal{C}$. Then $\Q(J[3^\infty]) \subseteq \san(\Q, 3)$.
\end{corollary}

\newcommand{\etalchar}[1]{$^{#1}$}

\end{CJK}
\end{document}